\documentclass[11pt,reqno]{amsart}
\usepackage{amsmath}
\usepackage{amssymb}
\usepackage{amstext}
\usepackage{amsfonts}
\usepackage{amsthm}
\usepackage{ifthen}
\usepackage{xspace}
\usepackage{comment}
\usepackage{graphicx}
\usepackage{mathrsfs}
\usepackage{enumerate}
\usepackage{breakcites}
\usepackage{xcolor}

\numberwithin{equation}{section}  

\newtheorem{punkt}{}[section]

\theoremstyle{plain}
\newtheorem{corollary}[punkt]{Corollary}
\newtheorem{lemma}[punkt]{Lemma}
\newtheorem{proposition}[punkt]{Proposition}
\newtheorem{theorem}[punkt]{Theorem}

\theoremstyle{definition}
\newtheorem{remark}[punkt]{Remark}

\newtheorem{examples}[punkt]{Examples}

\newtheorem{definition}[punkt]{Definition}

\theoremstyle{plain}
\newtheorem*{corollary*}{Corollary}
\newtheorem*{lemma*}{Lemma}
\newtheorem*{proposition*}{Proposition}
\newtheorem*{theorem*}{Theorem}

\theoremstyle{definition}
\newtheorem*{remark*}{Remark}
\newtheorem*{remarks*}{Remarks}
\newtheorem*{example*}{Example}
\newtheorem*{examples*}{Examples}
\newtheorem*{problem*}{Problem}
\newtheorem*{problems*}{Problems}
\newtheorem*{question*}{Question}
\newtheorem*{questions*}{Questions}
\newtheorem*{definition*}{Definition}
\newtheorem*{conjecture*}{Conjecture}
\newtheorem*{assumption*}{Assumption}
\newtheorem*{assumptions*}{Assumptions}
\newtheorem*{construction*}{Construction}

\def\mycmplx{\mathbb{C}}

\def\mynat{\mathbb{N}}

\def\myreal{\mathbb{R}}

\def\mylebesgue{\lambda \mskip -8mu \lambda}

\def\myd{\mathcal{D}}

\def\myh{\mathcal{H}}

\def\myl{\mathcal{L}}

\def\myo{\mathcal{O}}

\def\myq{\mathcal{Q}}

\def\mys{\mathcal{S}}

\def\A{\mathbf{A}}
\def\B{\mathbf{B}}
\def\C{\mathbf{C}}

\def\F{\mathbf{F}}
\def\G{\mathbf{G}}
\def\H{\mathbf{H}}
\def\I{\mathbf{I}}
\def\J{\mathbf{J}}
\def\M{\mathbf{M}}

\def\R{\mathbf{R}}
\def\T{\mathbf{T}}
\def\U{\mathbf{U}}
\def\V{\mathbf{V}}
\def\W{\mathbf{W}}
\def\X{\mathbf{X}}
\def\Y{\mathbf{Y}}
\def\Z{\mathbf{Z}}
\def\DDelta{\mathbf{\Delta}}

\def\e{\mathbf{e}}

\def\v{\mathbf{v}}

\def\ee{\mathbb{E}}
\def\pp{\mathbb{P}}

\newcommand{\E}{\mathbb{E}}

\def\rank{\qopname\relax{no}{rank}}
\def\trace{\qopname\relax{no}{trace}}

\def\re{\qopname\relax{no}{Re}\,}
\def\im{\qopname\relax{no}{Im}\,}

\def\eg{e.g.\@\xspace}
\def\ie{i.e.\@\xspace}

\def\i{{\operatorname{i}}}

\begin{document}

\renewcommand{\thefootnote}{}

\title[Sums of Products of Non-Hermitian Random Matrices]{Limiting Spectral Distributions \\ of Sums of Products \\ of Non-Hermitian Random Matrices}
\author[H.~K\"osters]{H.~K\"osters$^{1,3}$}
% \address{Department of Mathematics, Bielefeld University, Germany}
\author[A.~Tikhomirov]{A.~Tikhomirov$^{2,3}$}
% \address{Department of Mathematics, Komi Science Center of Ural Division of RAS, Syktyvkar State University, Russia}

\begin{abstract}
For fixed $l,m \ge 1$,
let $\X_n^{(0)},\X_n^{(1)},\dots,\X_n^{(l)}$ be independent random $n \times n$ matrices
with independent entries, let $\F_n^{(0)} := \X_n^{(0)} (\X_n^{(1)})^{-1} \cdots (\X_n^{(l)})^{-1}$,
and let $\F_n^{(1)},\hdots,\F_n^{(m)}$ be independent random matrices of the same form as $\F_n^{(0)}$.
We~investigate the limiting spectral distributions of the matrices 
$\F_n^{(0)}$ and $\F_n^{(1)} + \hdots + \F_n^{(m)}$ as $n \to \infty$. 
Our main result shows that the sum $\F_n^{(1)} + \hdots + \F_n^{(m)}$ has the same limiting eigenvalue distribution as $\F_n^{(0)}$
after appropriate rescaling. This extends recent findings by Tikhomirov and Timushev (2014).
Furthermore, we show that the limiting eigenvalue distribution of the matrices $\F_n^{(0)}$ 
is stable with respect to a suitably defined convolution $\oplus$.
%Furthermore, we propose a convolution $\oplus$ on a certain set
%of rotation-invariant probability measures on the complex plane such that 
%the limiting eigenvalue distribution of the matrices $\F_n^{(0)}$ is stable 
%with respect to the convolution $\oplus$.

To obtain our results, we apply the general framework 
recently introduced in G\"otze, K\"osters and Tikhomirov (2014)
to sums of products of independent random matrices and their inverses.
We establish the universality of the limiting singular value and eigenvalue distributions,
and we provide a closer description of the limiting distributions
in terms of free probability theory.
\end{abstract}

\keywords{non-hermitian random matrices; limiting spectral distributions; free probability theory;
stable distributions.\\[-7pt]}

\maketitle

% \footnotetext{}
\footnotetext{1) Department of Mathematics, Bielefeld University, Germany}
\footnotetext{2) Department of Mathematics, Komi Science Center of Ural Division of RAS, \\ Syktyvkar State University, Russia}
\footnotetext{3) Research supported by CRC 701.}

\medskip

%\begin{small}
%\tableofcontents
%\end{small}

\section{Introduction and Summary}
\label{sec:introduction}

The investigation of the asymptotic spectral distributions of random matrices
with independent entries is a major topic in random matrix theory. In recent years 
sums and products of independent \emph{non-Hermitian} random matrices 
with indepen\-dent entries have found increasing attention;
% 
% see e.g.~\cite{ARRS:2013,AGT:2010,Bordenave:2011,Burda:2013,BJW:2010,BJLNS:2011,
% Forrester:2014,FL:2014,GKT:2014,GT:2010b,J:2011,MNPZ:2014,OS:2011,Rogers:2010,
% Tikhomirov:2013,TT:2013,TT:2014} for results on \emph{global} spectral distributions,
% and also \linebreak the survey paper \cite{AI:2015} and the references therein
% for results on \emph{local} spectral distri\-butions.
% 
see e.g.~\cite{ARRS:2013,AB:2012,ABK:2014,AGT:2010,BD:2013,
Bordenave:2011,BJW:2010,BJLNS:2011,Forrester:2014,FL:2014,GKT:2014,GT:2010b,J:2011,
KS:2014,KZ:2014,LWZ:2014,MNPZ:2014,OS:2011,Rogers:2010,Tikhomirov:2013,TT:2013,TT:2014},
and also the survey papers \cite{AI:2015,Burda:2013} and the references therein.
In particular, the paper \cite{GKT:2014} provides a~general approach 
for the investigation of the limiting (global) spectral distributions
of products of independent random matrices with independent entries.
Furthermore, the paper \cite{TT:2014} shows that this approach proves useful 
for the investigation of sums of products of indepen\-dent random matrices as~well.
The main aim of the present paper is to~show that 
certain products of independent random matrices
give rise to random matrices with \emph{stable} limiting eigenvalue distributions,
in the sense that the sums of several independent copies of these products 
have the same limiting eigenvalue distribution after appropriate rescaling.

\pagebreak[2]

Throughout this paper, for each $n \geq 1$, let $\X_n^{(1)},\X_n^{(2)},\X_n^{(3)},\dots$ 
be independent random matrices of size $n \times n$ with independent entries.
More precisely, we~assume that 
\begin{align}
\label{eq:girko-ginibre}
\X_n^{(q)} = (\tfrac{1}{\sqrt{n}} X_{jk}^{(q)})_{jk=1,\hdots,n} \,,
\end{align}
where $(X_{jk}^{(q)})_{j,k,q \in \mynat}$ is a family of independent real~or~complex 
random variables such~that 
\begin{align}
\label{eq:real-moments}
\ee X_{jk}^{(q)} = 0, \ \ee (X_{jk}^{(q)})^2 = 1 \quad\text{in the real case}
\end{align}
and
\begin{align}
\label{eq:complex-moments}
\ee X_{jk}^{(q)} = 0, \ \ee (X_{jk}^{(q)})^2 = 0, \ \ee |X_{jk}^{(q)}|^2 = 1 \quad\text{in the complex case} \,,
\end{align}
and we additionally assume that this family is uniformly square-integrable, i.e.\@
\begin{align}
\label{eq:uniform-integrability}
\lim_{a \to \infty} \sup_{j,k,q \in \mynat} \ee \big( |X_{jk}^{(q)}|^2 \ \pmb{1}_{\{ |X_{jk}^{(q)}| \geq a \}} \big) = 0 \,.
\end{align}
In this case we also say the matrices $\X_n^{(q)}$ are independent \emph{Girko--Ginibre matrices}.
In the special case where the entries have real or complex Gaussian distributions,
we usually write $\Y_n^{(q)} = (\tfrac{1}{\sqrt{n}} Y_{jk}^{(q)})_{jk=1,\hdots,n}$ 
instead of $\X_n^{(q)} = (\tfrac{1}{\sqrt{n}} X_{jk}^{(q)})_{jk=1,\hdots,n}$ and call 
the matrices $\Y_n^{(q)}$ \emph{Gaussian random matrices} or \emph{Ginibre matrices}.
Note that the assumption \eqref{eq:uniform-integrability} is clearly satisfied
in this special case, the random variables $Y^{(q)}_{jk}$ being i.i.d.
% (independent and identically distributed)

We are interested in the spectral distributions 
of sums of products of the matrices $\X_n^{(q)}$ and their inverses,
such as 
$\X_n^{(1)} \X_n^{(2)} + \X_n^{(3)} \X_n^{(4)}$
or
$\X_n^{(1)} (\X_n^{(2)})^{-1} + \X_n^{(3)} (\X_n^{(4)})^{-1}$,
in~the limit~as $n \to \infty$. 
More precisely, we consider matrices of the form
\begin{align}
\label{eq:sums-of-products}
\F_n := \sum_{q=1}^{m} \F_n^{(q)} := \sum_{q=1}^{m} \prod_{r=1}^{l} (\X_n^{((q-1)l+r)})^{\varepsilon_{r}} \,,
\end{align}
where $m \in \mynat$,  $l \in \mynat$, and $\varepsilon_1,\hdots,\varepsilon_l \in \{ +1,-1 \}$ are fixed.
(Thus, the matrices $\F_n^{(q)}$ are independent random matrices of the same form
as the matrix $\prod_{r=1}^{l} (\X_n^{(r)})^{\varepsilon_{r}}$.)
Let us note that under the~above assumptions 
\eqref{eq:girko-ginibre} -- \eqref{eq:uniform-integrability},
for fixed $r \in \mynat$, \linebreak[2]
$\X_n^{(r)}$ is invertible with probability $1+o(1)$ as $n \to \infty$
(see \eg Lemma \ref{lemma:C1}), so~that $\F_n$ is well-defined
with probability $1+o(1)$ as $n \to \infty$.
Since we are interested in limit theorems in~probability,
this is sufficient for our purposes.

Furthermore, write $\lambda_1,\hdots,\lambda_n$ for the eigenvalues of $\F_n$
and $\mu_n := \tfrac1n \sum_{j=1}^{n} \delta_{\lambda_j}$ 
for the \emph{(empirical) eigenvalue distribution} of $\F_n$.
Note that $\mu_n$ will in general be a probability measure
on the complex plane, $\F_n$~being non-Hermitian.
We are inter\-ested in the problem(s) whether there exists 
a non-random probability measure $\mu$ on the complex plane
such that $\mu_n \to \mu$ weakly (in probability, say)
and whether this probability measure $\mu$ can be described more explicitly
(in terms of its density or one of its transforms, say).
If existent, the probability measure $\mu$ is also called
the \emph{limiting eigenvalue distribution} of the matrices $\F_n$.

% \pagebreak[2]

As is well known in random matrix theory,  
the limiting eigenvalue distribution $\mu$ is usually \emph{universal},
i.e.\@ it does not depend on the distributions of the matrix entries
apart from a~few moment conditions \pagebreak[2]
as in \eqref{eq:real-moments} -- \eqref{eq:uniform-integrability}.
In our situation, we have the following universality result:

\begin{proposition}
\label{prop:simple-limit-theorem}
Let the matrices $\F_n$ be defined as in \eqref{eq:sums-of-products}.
Then there exists \linebreak a non-random probability measure $\mu$
on $\mycmplx$ such that $\mu_n \to \mu$ weakly in~prob\-ability,
and the limiting eigenvalue distribution $\mu$ is same as in the Gaussian case,
i.e.\@ for the corresponding matrices $\F_n$ derived from Gaussian random matrices 
$\Y_n^{(1)},\Y_n^{(2)},\Y_n^{(3)},\hdots.$
\end{proposition}

For a closer description of the limiting eigenvalue distribution $\mu$,
see Theorem~\ref{thm:limit-theorem} below. For the moment, we confine ourselves
to a few comments.

First of all, by universality, 
it remains to find the limiting eigenvalue distribution $\mu$ in the Gaussian case.
Now, the Gaussian random matrices $\Y_n^{(r)}$ are \emph{bi-unitary invariant},
i.e.\@ for any unitary matrices $\U_n^{(r)}$ and $\V_n^{(r)}$
of size $n \times n$, the matrices $\U_n^{(r)} \Y_n^{(r)} \V_n^{(r)}$ 
have the same (matrix-valued) distributions as the matrices $\Y_n^{(r)}$. 
This clearly implies that the limiting eigenvalue distribution $\mu$ of the matrices $\F_n$, if~existent, 
will be a~\emph{rotation-invariant} probability measure on the complex plane, 
i.e.\@ for any $u \in \mathbb{T}$ $:= \{ z \in \mathbb{C} : |z|=1 \}$,
the induced probability measure of $\mu$ under the mapping $z \mapsto uz$
will coincide with~$\mu$. Hence, if $\mu$ has a density $f$
(which will be the case in all our examples),
this density may be supposed to be rotation-invariant as well,
and we write $f(r)$ instead of $f(z)$, with $r := |z|$.

\pagebreak[2]

Let us mention some relevant results from the literature.

\begin{examples} \
\label{ex:theexamples}

(a) (\emph{Circular Law}) \ 
Let $\F_n = \X_n^{(1)}$. Then $f(r) := \frac{1}{\pi} \, \pmb{1}_{[0,1]}(r)$,
i.e.\@ $\mu$ is the uniform distribution on the unit disk.

(b) Let $\F_n = \X_n^{(1)} + \hdots + \X_n^{(m)}$.
Then $\F_n$ is a random matrix with independent entries
of mean $0$ and variance $m/n$, so, by simple rescaling, 
$f(r) := \tfrac{1}{m\pi} \, \pmb{1}_{[0,\sqrt{m}]}(r)$.
In~particular, for the rescaled matrices $\tfrac{1}{\sqrt{m}} \F_n$,
the~limiting eigenvalue distribution is again the uniform distribution on the unit disk.

(c) 
Let $\F_n = \X_n^{(1)} \X_n^{(2)}$. 
Then $f(r) = \frac{1}{2\pi r} \, \pmb{1}_{[0,1]}(r)$,
i.e.\@ $\mu$ is the induced dis\-tribution of the uniform distribution
on the unit disk under the mapping $z \mapsto z^2$.
See e.g.\@ \cite[Section 8.2.2]{GKT:2014} for a `simple' derivation.

(d)
Let $\F_n = \X_n^{(1)} \X_n^{(2)} + \hdots + \X_n^{(2m-1)} \X_n^{(2m)}$. 
Then $f(r) = \frac{1}{\pi \sqrt{(m-1)^2+4r^2}} \, \pmb{1}_{[0,\sqrt{m}]}(r)$;
see \cite[Section 2]{TT:2014}.

(e) (\emph{Spherical Law}) \ 
Let $\F_n = \X_n^{(1)} (\X_n^{(2)})^{-1}$. Then $f(r) = \frac{1}{\pi(1+r^2)^2}$,
i.e. \linebreak $\mu$ is the spherical distribution on the complex plane
(which is, by definition, the induced distribution of the uniform distribution
on the $2$-dimensional sphere $S_2 := \{ x \in \mathbb{R}^3 : \|x\|=1 \}$
under stereographic projection to the complex plane $\mathbb{C} \simeq \mathbb{R}^2$).

(f) 
Let $\F_n = \X_n^{(1)} (\X_n^{(2)})^{-1} + \hdots + \X_n^{(2m-1)} (\X_n^{(2m)})^{-1}$. 
Then $f(r) = \frac{m^2}{\pi (m^2+r^2)^2}$; see~\cite[Section 3]{TT:2014}.
Thus, for the rescaled matrices $\tfrac{1}{m} \F_n$,
the~limiting eigenvalue distribution is again the spherical distribution on the complex plane.

\end{examples}

In view of examples (b) and (f), it seems natural to ask whether there exist
further examples of random matrices $\F_n^{(0)}$ such that for any $m \in \mynat$,
the sums of $m$ independent matrices of the same form as $\F_n^{(0)}$ 
have the same limiting eigenvalue distribution as the original 
random matrices $\F_n^{(0)}$, \pagebreak[2] after appropriate rescaling. 
We~will answer this question in the affirmative by proving the following result, 
which contains examples (b) and (f) as special cases:

\begin{theorem}
\label{thm:new}
Fix $m \in \mynat$ and $l \in \mynat_0$, let 
\begin{align}
\label{eq:F1L}
\F_n^{(0)} := (\X_n^{(0)}) (\X_n^{(1)})^{-1} \cdots (\X_n^{(l)})^{-1} \,,
\end{align}
where $\X_n^{(0)},\X_n^{(1)},\dots,\X_n^{(l)}$ are independent random matrices 
as in \eqref{eq:girko-ginibre} -- \eqref{eq:uniform-integrability},
and let $\F_n^{(1)},\hdots,\F_n^{(m)}$ be independent matrices
of the same form as $\F_n^{(0)}$.
Then the~matrices $m^{-(l+1)/2} (\F_n^{(1)} + \hdots + \F_n^{(m)})$ and $\F_n^{(0)}$
have the same limiting eigenvalue distribution $\mu$.
More precisely, we have $\mu = \myh(\sigma_s(\frac{2}{l+1}))$,
where $\sigma_s(\frac{2}{l+1})$ is the symmetric $\boxplus$-stable distribution 
with parameter $\frac{2}{l+1}$ (see Section~\ref{sec:boxplus-stable})
and $\myh(\sigma_s(\frac{2}{l+1}))$ is the associated rotation-invariant distribution on $\mycmplx$
(see Section~\ref{sec:FPT}).
\end{theorem}

Let us note that for $l = 0$ and $l = 1$, we re-obtain the above-mentioned results
from \textsc{Tikhomirov} and \textsc{Timushev} \cite{TT:2014}.
Moreover, as we will see in Section~3, 
apart from a possible permutation of the exponents $\pm 1$,
the matrices $\F_n^{(0)}$ in Theorem~\ref{thm:new} are the only examples
of products of \emph{independent} Girko--Ginibre matrices and their inverses
such that for any $m \in \mynat$, $\F_n^{(0)}$ and $\F_n^{(1)} + \cdots + \F_n^{(m)}$ 
have the same limiting eigenvalue distribution after appropriate rescaling.
In particular, the matrices % it turns out that the~matrices
\begin{align}
\label{eq:FKL}
\F_n^{(0)} := \X_n^{(1)} \cdots \X_n^{(k)} (\X_n^{(k+1)})^{-1} \cdots (\X_n^{(k+l)})^{-1}
\end{align}
with $k > 1$ do not share this property.

\pagebreak[2]

However, the same limiting eigenvalue distributions 
may arise for products involving \emph{powers} of random matrices:

\begin{theorem}
\label{thm:newpower}
Fix $m \in \mynat$, $k \in \mynat_0$ and $l_1,\hdots,l_k \in \mynat$, let $l := l_1 + \hdots + l_k$ and
\begin{align}
\F_n^{(0)} := (\X_n^{(0)}) (\X_n^{(1)})^{-l_1} \cdot (\X_n^{(k)})^{-l_k} \,,
\end{align}
where $\X_n^{(0)},\X_n^{(1)},\dots,\X_n^{(k)}$ are independent random matrices 
as in \eqref{eq:girko-ginibre} -- \eqref{eq:uniform-integrability},
and let $\F_n^{(1)},\hdots,\F_n^{(m)}$ be independent matrices
of the same form as $\F_n^{(0)}$.
Then the~matrices $m^{-(l+1)/2} (\F_n^{(1)} + \hdots + \F_n^{(m)})$ and $\F_n^{(0)}$
have the same limiting eigenvalue~distribution $\mu$,
which is the same as in Theorem \ref{thm:new}.
\end{theorem}

To obtain the preceding results, we~apply the general framework from \cite{GKT:2014}
for the investigation of (global) limiting spectral distributions
to sums of products of~independent Girko--Ginibre random matrices and their inverses
(see Section~\ref{sec:limit-theorem}).
Related results for various special cases can be found e.\,g. in 
\cite{ARRS:2013,AGT:2010,Bordenave:2011,Burda:2013,BJW:2010,Forrester:2014,
GT:2010b,J:2011,MNPZ:2014,OS:2011,Tikhomirov:2013,TT:2014}.
In particular, in the Gaussian case, the limiting eigen\-value and singular value distributions
of the products \eqref{eq:FKL} were recently obtained in \cite{ARRS:2013} and \cite{Forrester:2014},
respectively.

Furthermore, to identify the limiting spectral distributions,
we use tools from free probability theory.
Here it is worth emphasizing that for the matrices occur\-ring in Theorems \ref{thm:new} and \ref{thm:newpower}
the limiting spectral distributions may be described relatively explicitly.
It seems that comparable results are available only in a~few special cases, 
see e.\,g. \cite{BL:2001,HL:2000,HS:2007,Lehner:2001}.
Let us mention, however, the very recent work \cite{BMS:2013,Speicher:2015,BSS:2015}
which provides an algorithm for calculating the Brown measures 
of general polynomials in free non-commutative random variables.
This yields many further examples where the limiting spectral distributions 
may now be determined.

We present formal proofs of Theorems \ref{thm:new} and \ref{thm:newpower}
in Section~\ref{sec:formalproof},
after recalling some relevant facts from random matrix theory
and free probability theory in Section~\ref{sec:background}.
In Section~\ref{sec:convolution}, we point out that 
the limiting spectral distributions in Theorem \ref{thm:new}
may be viewed as \emph{stable} distributions
with respect to an appropriate convolution $\oplus$
defined on a certain set of rotation-invariant probability measures on the complex plane,
and we discuss some generalizations of Theorem~\ref{thm:new}.
Section~\ref{sec:limit-theorem} is~devoted to the proof of an enhanced version
of Proposition \ref{prop:simple-limit-theorem}, viz.~Theorem~\ref{thm:limit-theorem}. 
In particular, combining the formal proofs from Section \ref{sec:formalproof}
with that of Theorem~\ref{thm:limit-theorem}, we obtain rigorous proofs
of Theorems \ref{thm:new} and \ref{thm:newpower}; 
see Section \ref{sub:rigor} for details. 
Finally, in Section~\ref{sec:auxiliary-results}, 
we compile a number of auxiliary results from the~literature.

\medskip

\section{Background}
\label{sec:background}

In this section we collect some relevant concepts and results from the literature.

\subsection{Results from Random Matrix Theory.} \ 
\label{sec:RMT}

Given a random matrix $\F_n$ of dimension $n \times n$,
let $s_1(\F_n) \geq \cdots \geq s_n(\F_n)$ denote
the singular values of $\F_n$ (in decreasing order),
and let $\lambda_1(\F_n),\hdots,\lambda_n(\F_n)$ denote
the eigenvalues of $\F_n$ \linebreak[2] (in arbitrary order).
Then the probability measure
$$
\nu_n := \nu(\F_{n}) := \tfrac1n \sum_{j=1}^{n} \delta_{s_j(\F_n)}
$$
is called the (empirical) \emph{singular value distribution} of the matrix $\F_n$,
and if \linebreak there exists a non-random probability measure $\nu_\F$
such that $\nu(\F_{n}) \to \nu_\F$ weakly in probability,
we call $\nu_\F$ the \emph{limiting singular value distribution}
of the matrices $\F_n$.
Similarly, the~probability measure
$$
\mu_n := \mu(\F_{n}) := \tfrac1n \sum_{j=1}^{n} \delta_{\lambda_j(\F_n)}
$$ 
is called the (empirical) \emph{eigenvalue distribution} of $\F_n$,
and if there exists a non-random probability measure $\mu_\F$
such that $\mu(\F_{n}) \to \mu_\F$ weakly in probability,
we call $\mu_\F$ the \emph{limiting eigenvalue distribution} of the~matrices $\F_n$.
Following \textsc{Girko}, we~will study the limiting eigenvalue distribution
of the matrices $\F_n$ by studying the limiting singular value distributions
of the \emph{shifted} matrices $\F_n - \alpha \I_n$, for all $\alpha \in \mathbb{C}$.
See Theorem \ref{thm:main} for a statement suited for our purposes.

\pagebreak[2]

In doing so, we will often consider the \emph{Hermitian} matrices
\begin{align}
\label{eq:hermitization}
\V_n := \left[ \begin{array}{cc} \mathbf{O} & \F_n^{} \\ \F_n^* & \mathbf{O} \end{array} \right] \qquad\text{and}\qquad \W_n := \F_n^{} \F_n^* \,.
\end{align}
Note that the eigenvalues of these matrices are given by
$\pm s_1,\hdots,\pm s_n$ and $s_1^2,\hdots,s_n^2$, respectively.
For this reason, the probability measures $\mu(\V_n)$ and $\mu(\W_n)$
will also be called the \emph{symmetrized} and \emph{squared} \emph{singular value distribution}
of the matrix $\F_n$, respectively.
It is easy to see that knowledge of one of the~distributions $\nu(\F_n),\mu(\V_n),\mu(\W_n)$
(or its convergence) implies knowledge of the~other two (or their convergence).
More precisely, if $\mys$ denotes the operator which associates
with each distribution $\nu$ on $(0,\infty)$
its symmetrization on~$\myreal^*$
and $\myq$ denotes the operator which associates
with each symmetric distribution $\mu$ on~$\myreal^*$
its~induced distribution on $(0,\infty)$ under the mapping $x \mapsto x^2$,
the operators $\mys$~and~$\myq$ are one-to-one,
and we have the relations
\begin{align}
\label{eq:S-and-Q}
\mu(\V_n) = \mys \nu(\F_n)
\quad\text{and}\quad
\mu(\W_n) = \myq \mu(\V_n) \,.
\end{align}

In the special case where $\F_n = \X_n^{(1)}$, $n \in \mynat$,
it is well-known that $\mu(\W_n) \to \gamma$ weakly in probability,
where
\begin{align}
\label{eq:marchenko-pastur}
\gamma(dx) = \frac{1}{2\pi} \sqrt{\frac{4-x}{x}} \, \pmb{1}_{(0,4)}(x) \, \mylebesgue(dx) \,.
\end{align}
This result is also known as the \emph{Marchenko--Pastur law} (with~parameter~1),
and the measure $\gamma$ is also called the \emph{Marchenko--Pastur distribution}
(with~parameter~1). 
Therefore, in the special case where $\F_n = (\X_n^{(1)})^{-1}$, $n \in \mynat$,
we~have $\mu(\W_n) \to \gamma^{-1}$ weakly in probability,
where $\gamma^{-1}$ is the induced measure of $\gamma$ 
under the mapping $x \mapsto x^{-1}$. \pagebreak[2]
We will call this measure the \emph{inverse Marchenko--Pastur distribution}.
The distributions $\gamma$ and $\gamma^{-1}$ will serve 
as building blocks for more complex results.
Also, for $t > 0$, let $\gamma^{+1}_t := \gamma^{+1} := \gamma$,
and let $\gamma^{-1}_t$ denote the induced measure of $\gamma$
under the mapping $x \mapsto (x+t)^{-1} x (x+t)^{-1}$.
(These notions are motivated \linebreak by our regularization procedure in Section~5.) 
Note that $\gamma^{+1}_t = \gamma$ for all $t > 0$, 
while $\gamma^{-1}_t \to \gamma^{-1}$ weakly as $t \to 0$.
Finally, let us note that the $S$-transforms of $\gamma$ and $\gamma^{-1}$
are given by
\begin{align}
\label{eq:S-marcenko-pastur}
S_\gamma(z) = \frac{1}{z+1} \qquad\text{and}\qquad S_{\gamma^{-1}}(z) = -z \,,
\end{align}
respectively, see e.g.\@ Section 8.1.1 in \cite{GKT:2014}.

\medskip

In the next theorem, the first part is a special case of a result which goes back to \textsc{Girko}
(see also \textsc{Bordenave} and \textsc{Chafai} \cite{BC:2012}), while the second part is taken 
from Section~7 in \cite{GKT:2014}. For $\alpha \in \mathbb{C}$, introduce the Hermitian matrix
\begin{align}
\label{eq:matrix-J}
\J_n(\alpha) := \left[ \begin{array}{cc} \mathbf{O} & -\alpha\I_n \\ -\overline\alpha\I_n & \mathbf{O} \end{array} \right] 
\end{align}
and the Bernoulli measure
\begin{align}
\label{eq:measure-B}
B(\alpha) := \tfrac12\delta_{-|\alpha|} + \tfrac12\delta_{+|\alpha|} \,,
\end{align}
and note that $B(\alpha)$ is the eigenvalue distribution of $\J_n(\alpha)$.

\medskip

\begin{theorem}[Convergence of Spectral Distributions]
\label{thm:main}
Let $(\F_n)_{n \in \mynat}$ be a sequence of random matrices,
let $(\V_n)_{n \in \mynat}$ be defined as in \eqref{eq:hermitization},
and suppose the following:
\begin{enumerate}[(a)]
\item
For each $n \in \mynat$, $\F_n$ has size $n \times n$.
\item
There exists a non-random probability measure $\mu_\V$ on $\myreal$ such that for all $\alpha \in \mathbb{C}$, 
$\mu(\V_n + \J_n(\alpha)) \to \mu_\V \boxplus B(\alpha)$
weakly in probability.
\item
The random matrices $\F_n$ satisfy the conditions $(C0)$, $(C1)$ and $(C2)$ from \cite{GKT:2014}.
(These conditions are the same as in Condition C in Section~5.)
\end{enumerate}
Then the empirical eigenvalue distributions of the matrices $\F_n$ converge weakly in~probability 
to a limit $\mu_\F$, where $\mu_\F$ is the unique probability measure on $\mathbb{C}$ 
such~that
\begin{align}
\label{eq:logpotential}
U_\F(\alpha) := - \! \int \log|z-\alpha| \, d\mu_\F(z) = - \! \int \log|x| \, d(\mu_\V \boxplus B(\alpha))(x)
\end{align}
for all $\alpha \in \mathbb{C}$.

\pagebreak[2]

Moreover, the probability measure $\mu_\F$ is rotation-invariant,
and with the notation from \cite{GKT:2014} and under regularity conditions, 
it has the Lebesgue density
\begin{align}
\label{eq:GKT}
f(u,v) = \frac{1}{2\pi |\alpha|^2} \left( u \, \frac{\partial \psi}{\partial u} + v \, \frac{\partial \psi}{\partial v} \right) \,,
\end{align}
where $\psi$ is a continuous function on $\mycmplx^*$ taking values in $[0,1]$
and satisfying the~equation
\begin{align}
\psi(\alpha) (1 - \psi(\alpha)) = - |\alpha|^2 (1-\psi(\alpha))^2 \Big( S_\V(-(1-\psi(\alpha))) \Big)^2 \,.
\end{align}
Here, $S_{\V}$ denotes the $S$-transform of $\mu_{\V}$.
Alternatively, we~may write
\begin{align}
\label{eq:alternative}
f(u,v)=\frac1{\pi|\alpha|^2}\frac{\psi(\alpha)(1-\psi(\alpha))}{1-2\psi(\alpha)(1-\psi(\alpha))\dfrac{S'_\V(-(1-\psi(\alpha)))}{S_\V(-(1-\psi(\alpha)))}} \,.
\end{align}
\end{theorem}

\medskip
\medskip

\subsection{Results from Free Probability Theory.} \ 
\label{sec:FPT}

To describe the limiting singular value distributions of the random matrices $\F_n$
in~Proposition \ref{prop:simple-limit-theorem}, we will use various concepts and results from free probability theory. 
See e.g.\@ \cite{DNV:1992,NS:2006} for a thorough introduction to free probability theory,
or Section~5 in \cite{GKT:2014} for a brief introduction tailored to our purposes.
In~particular, we~will use the free additive and multiplicative convolutions $\boxplus$~and~$\boxtimes$,
the asso\-ciated $R$ and $S$ transforms
(also for probability measures with \emph{unbounded} support),
and the asymptotic freeness of random matrices.
Furthermore, we will frequently use the following result:

\begin{proposition}[Asymptotic Freeness]
\label{prop:asymptoticfreeness}
For each $n \in \mynat$, let $\A_n$ and $\B_n$ be~in\-dependent bi-unitary invariant random matrices of size $n \times n$
such that 
$$
\sup_{n \in \mynat} \max \Big\{ \ee \Big( \tfrac1n \trace (\A_n^{} \A_n^*)^k \Big), \ee \Big( \tfrac1n \trace (\B_n^{} \B_n^*)^k \Big) \Big\} < \infty
$$
for all $k \in \mynat$,
and suppose that there exist compactly supported (deterministic) probability measures $\mu_{\A \A^*}$ and $\mu_{\B \B^*}$ on $(0,\infty)$
such that $\mu(\A_n \A_n^*) \to \mu_{\A \A^*}$ and $\mu(\B_n \B_n^*) \to \mu_{\B \B^*}$ weakly in probability.
% such that $\A_n \A_n^* \to \mu_{\A \A^*}$ and $\B_n \B_n^* \to \mu_{\B \B^*}$ in~moments.
\begin{enumerate}[(a)]
\item
The families $\{ \A_n,\A_n^* \}$ and $\{ \B_n,\B_n^* \}$ are asymptotically free, \\
and $(\A_n \B_n) (\A_n \B_n)^* \to \mu_{\A \A^*} \boxtimes \mu_{\B \B^*}$ in moments.
% \item
% $(\A_n + \B_n) (\A_n + \B_n)^* \to \myq (\myq^{-1} \nu_{\A \A^*} \boxplus \myq^{-1} \nu_{\B \B^*})$ in moments.
\item
For any $k,l \in \mynat$, the matrices $(\A_n^k)^* \A_n^k$ and $\A_n^l (\A_n^l)^*$ are asymptotically free, \\
and for any $k \in \mynat$, $\A_n^{k} (\A_n^{k})^* \to \mu_{\A \A^*}^{\boxtimes k}$ in moments.
\item
The matrices $\V_n(\A_n)$ and $\V_n(\B_n)$ are asymptotically free, \\
and $\V_n(\A_n) + \V_n(\B_n) \to \mu_{\V(\A)} \boxplus \mu_{\V(\B)}$ in moments.
\item
The matrices $\V_n(\A_n)$ and $\J_n(\alpha)$ are asymptotically free, \\
and $\V_n(\A_n) + \J_n(\alpha) \to \mu_{\V(\A)} \boxplus B(\alpha)$ in moments.
\end{enumerate}
Here, $\V(\A_n)$ and $\V(\B_n)$ are defined similarly as in Eq.~\eqref{eq:hermitization},
and $\mu_{\V(\A)}$~and~$\mu_{\V(\B)}$ denote the corresponding limiting distributions.
\end{proposition}

Here, parts (a) and (b) follow from the results in Section 4.3 in \cite{HP:2000},
part (d) is proved in Section~5 in \cite{GKT:2014}, and part (c) follows from similar arguments.
Also, let~us~mention that part (c) is already implicit in \cite{TT:2014}.

\begin{remark}
Let us mention that Proposition \ref{prop:asymptoticfreeness} may be used
to establish the~weak convergence of the \emph{mean} singular value distributions
of the matrices $\A_n \B_n$, $\A_n^k$ and $\A_n + \B_n$.
However, in most of the situations in which we will use Proposition \ref{prop:asymptoticfreeness} later,
this already implies the weak convergence in probability of the singular value distributions
of these matrices, see e.g.\@ Section~A.1 in~\cite{GKT:2014}.
% see the variance estimate in Section~A.1 in~\cite{GKT:2014}.
\end{remark}

It is worth pointing out that there is another description of the density $f(u,v)$
of the limiting eigenvalue distribution in Theorem \ref{thm:main}. This description is 
due to \textsc{Haagerup} and \textsc{Larsen} \cite{HL:2000} 
for probability measures with bounded support 
and to \textsc{Haagerup} and \textsc{Schultz} \cite{HS:2007} \pagebreak[2]
for probability measures with unbounded support. 
Actually, in these papers, the density $f$ is shown to~describe the \emph{Brown~measure} 
of a so-called \emph{$R$-diagonal element} in a $W^*$-probability space. Roughly speaking, 
an $R$-diagonal element is a non-commutative random variable of the form $uh$, 
\linebreak where $u$ is a~Haar unitary and $h$ is positive element $*$-free from $u$.

For our purposes, this description of the density $f$ may be summarized as~follows:
In the situation of Theorem \ref{thm:main},
suppose that the matrices $\W_n := \F_n \F_n^*$ have a~limiting distribution $\mu_{\W}$
on $(0,\infty)$ which is not a Dirac measure and for which
\begin{align}
\label{eq:log-integrability-1}
\int \log^+ t \ d\mu_{\W}(t) < \infty \,.
\end{align}
Let $S_{\W}$ denote the $S$-transform of $\mu_{\W}$, and set 
$$
F(t) := \frac{1}{\sqrt{S_\W(t-1)}} \,.
$$
Then $F$ is a smooth bijection from the interval $(0,1)$ to the interval 
$$
(a,b) := \left( \Big( \int x^{-2} \, d\nu_\F(x) \Big)^{-1/2}, \Big( \int x^{2} \, d\nu_\F(x) \Big)^{1/2} \right)
$$
(where $1/\infty := 0$ and $1/0 := \infty$),
and the limiting eigenvalue distribution $\mu_\F$ of the~matrices $\F_n$
has a rotation-invariant density $f(r)$ given by 
\begin{align}
\label{eq:haagerup}
f(r) = \frac{1}{2\pi r \, F'(F^{-1}(r))} \pmb{1}_{(a,b)}(r) \,.
\end{align}
Clearly, the connection to the above Theorem \ref{thm:main} arises from the fact 
that $\psi = F^{-1}$ on the interval $(a,b)$. % as is easily verified. (at least formally).

Furthermore, Equation \eqref{eq:haagerup} shows that $F^{-1}(r) = \int_0^r 2\pi s \, f(s) \, ds$,
which implies that $\mu_\W$ is uniquely determined by $\mu_\F$.
Thus, we get a one-to-one correspondence between the set of all distributions 
$\mu_{\W}$ on $(0,\infty)$ satisfying \eqref{eq:log-integrability-1}
and a certain set $H$ of rotation-invariant distributions on $\mycmplx$.
Composing this correspondence with the~operator $\mathcal{Q}$
introduced above \eqref{eq:S-and-Q},
we obtain a one-to-one correspondence $\myh$ between 
the set of all symmetric distributions $\mu_{\V}$ on $\myreal^*$
such that % $\int \log^+|s| \, d\mu(s) < \infty$
\begin{align}
\label{eq:log-integrability-2}
\int \log^+ |t| \ d\mu_{\V}(t) < \infty
\end{align}
and the above-mentioned set $H$ of rotation-invariant distributions on $\mycmplx$.
It is easy to check that for any symmetric probability measure $\mu$ on $\myreal^*$
satisfying \eqref{eq:log-integrability-2}, we~have
\begin{align}
\label{eq:scaling}
\myh(\myd_c \mu) = \myd_c \myh(\mu)
\end{align}
for all $c > 0$, where $\myd_c$ is the scaling operator 
which maps a probability measure on $\myreal$ or $\mycmplx$
to its induced measure under the mapping $x \mapsto cx$.

\pagebreak[2]

For random matrices $\F_n$ and $\V_n$ as in Theorem \ref{thm:main},
the correspondence $\myh$ describes the relationship
between the limiting spectral distributions $\mu_\V$~and~$\mu_\F$,
i.e.\@ we have the relation
\begin{align}
\label{eq:H}
\mu_\F = \mathcal{H}(\mu_\V) \,.
\end{align}

\medskip

\subsection{Results on $\boxplus$-Stable Distributions.} \ 
\label{sec:boxplus-stable}

Let us collect some results on $\boxplus$-stable distributions which will be needed later.
% Recall that
A distribution $\mu$ on $\myreal$ is called \emph{(strictly) $\boxplus$-stable}
if there exists a~constant $\alpha > 0$ such that $\mu^{\boxplus m} = \myd_{m^{1/\alpha}} \mu$ 
for all $m \in \mynat$. 
Here, $\myd_c$ is defined as in Equation \eqref{eq:scaling}.
We will often call the constant $\alpha$
the \emph{stability index} of the~$\boxplus$-stable distribution $\mu$.

The (strictly) $\boxplus$-stable distributions have been investigated
in \cite{BV:1993}, \cite{BP:1999} and \cite{AP:2009}.
First of all, let us recall that for any $\boxplus$-stable distribution, 
$\alpha \in {]}0,2{]}$. We~will need the following result, 
which is contained in \cite[Appendix A]{BP:1999} and~\cite{AP:2009}:

\begin{proposition}
\label{prop:symmetric-free-stable}
Fix $\alpha \in {]}0,2{]}$.
For a symmetric probability measure $\mu$ on $\myreal^*$,
the following are equivalent:
\begin{enumerate}[(i)]
\item
$\mu$ is (strictly) $\boxplus$-stable with stability index $\alpha$.
\item
$R_\mu(z) = bz^{\alpha - 1}$,
where $b \in \mycmplx^*$ with $\arg b = -\pi + \alpha \pi/2$.
\item
$S_\mu(z) = z^{(1/\alpha)-1} / b^{1/\alpha}$,
where $b \in \mycmplx^*$ with $\arg b = -\pi + \alpha \pi/2$.
\end{enumerate}
Moreover, in this case, the constants $b$ in parts (ii) and (iii) are the~same.
\end{proposition}

Here, for the $S$-transform $S_\mu(z)$, we make the convention that we take
arguments in $]{-}\pi,{+}\pi]$ to define powers of~$b$
and arguments in $(-2\pi,0)$ to define powers of~$z$.
Thus, we have $S_\mu(z) \in (0,\infty)\i$ when $z \in (-1,0)$,
in line with the convention in~\cite{GKT:2014}.

Henceforward, we write $\sigma_s(\alpha)$ for the (unique) symmetric $\boxplus$-stable distribution
with parameters $\alpha \in {]}0,2{]}$ and $b := e^{(-\pi + \alpha\pi/2)\i}$.
Note that in the special cases $\alpha = 2$ and $\alpha = 1$,
we obtain the standard semi-circle and Cauchy distribution, respectively.
Furthermore, let us recall from \cite[Appendix~A]{BP:1999} that
the distribution $\sigma_s(\alpha)$ has a continuous density $f_\alpha$ 
such that $f_\alpha(x) = \myo(|x|^{-\alpha-1})$ as $|x| \to \infty$.
Thus, in particular, the distribution $\sigma_s(\alpha)$ satisfies
Condition \eqref{eq:log-integrability-2}.

\medskip

\section{Formal Proof of Theorems \ref{thm:new} and \ref{thm:newpower}}
\label{sec:formalproof}

In this section we give a formal proof of Theorem \ref{thm:new} in the Gaussian case 
which~conveys the main idea without being cluttered by technical details.
Indeed, by Proposition \ref{prop:simple-limit-theorem}, once Theorem \ref{thm:new} is proved 
in the Gaussian case, it~follows that the result continues to hold in the general case.

The proof of Theorem \ref{thm:new} presented here is formal 
in that we use the concept of \emph{asymptotic freeness} for random matrices 
whose `moments' (in the sense of free probability theory) do not exist.
Moreover, we use Theorem \ref{thm:main} purely formally without checking the assumptions.
However, the argument can easily be converted into a~rigorous proof 
by using the regularization procedure and the results from Section~5;
see Section~\ref{sub:rigor} for more comments.
Similar remarks pertain to the~proof of Theorem \ref{thm:newpower}.

\pagebreak[2]

\begin{proof}[Formal Proof of Theorem \ref{thm:new}]
Recall that we concentrate on the Gaussian case,
\ie we have $\F_n^{(0)} = \Y_n^{(0)} (\Y_n^{(1)})^{-1} \cdots (\Y_n^{(l)})^{-1}$,
where $\Y_n^{(0)},\Y_n^{(1)},\hdots,\Y_n^{(l)}$ are inde\-pendent random matrices of dimension $n \times n$ 
with independent real or complex Gaussian entries with mean $0$ and variance $1/n$.

% \pagebreak[2]

First consider the case $m = 1$, so that $\F_n = \F_n^{(0)}$.
By the results from Section~2.1, we know that the limiting eigenvalue distributions of the matrices 
$\Y_n^{(0)} (\Y_n^{(0)})^*$ and $(\Y_n^{(r)})^{-1} ((\Y_n^{(r)})^{-1})^*$ ($r \ne 0$) 
are given by $\gamma$ and $\gamma^{-1}$, respectively, 
with $S$-trans\-forms given by \eqref{eq:S-marcenko-pastur}.
Since $\Y_n^{(0)},(\Y_n^{(1)})^{-1},\dots,(\Y_n^{(l)})^{-1}$ 
are independent bi-unitary invariant matrices, 
it follows by ``asymptotic freeness'' (see Proposition \ref{prop:asymptoticfreeness}\,(a))
\linebreak that the limiting eigenvalue distribution of the matrices $\W_n$ 
from \eqref{eq:hermitization} is given by~the~$S$-transform
$$
S_\W(z) = \frac{(-z)^l}{z+1} \,.
$$
Thus, the limiting eigenvalue distribution of the matrices $\V_n$
from \eqref{eq:hermitization} is given by~the~$S$-transform
$$
S_\V(z) = \sqrt{\frac{z+1}{z} \frac{(-z)^l}{z+1}} = \i^l z^{(l-1)/2} \,.
$$
In view of Proposition \ref{prop:symmetric-free-stable},
the corresponding distribution is $\sigma_s(\frac{2}{l+1})$,
the symmetric $\boxplus$-stable distribution of parameter $\frac{2}{l+1}$.
Hence, again by ``asymptotic freeness'' (see~Proposition \ref{prop:asymptoticfreeness}\,(d)),
we find that for any $\alpha \in \mathbb{C}$, the limiting eigenvalue distribution of the matrices $\V_n + \J_n(\alpha)$
is given by $\sigma_s(\frac{2}{l+1}) \boxplus B(\alpha)$.
It therefore follows from Theorem \ref{thm:main} 
and Equation \eqref{eq:H} that the limiting eigenvalue distribution 
of the matrices $\F_n$ is given by $\myh(\sigma_s(\frac{2}{l+1}))$.

% \pagebreak[2]

Now consider the case $m > 1$. Here we use that
if $\F_n^{(1)},\hdots,\F_n^{(m)}$ are independent bi-unitary invariant random matrices, 
their Hermitizations $\V_n^{(1)},\hdots,\V_n^{(m)}$ (defined as in \eqref{eq:hermitization}) 
are ``asymptotically free'' (see Proposition \ref{prop:asymptoticfreeness}\,(c)).
Thus, the~matrices
$$
\widetilde\F_n := m^{-(l+1)/2} (\F_n^{(1)} + \hdots + \F_n^{(m)})
$$
have the Hermitizations
$$
\widetilde\V_n := m^{-(l+1)/2} (\V_n^{(1)} + \hdots + \V_n^{(m)}) \,,
$$
with limiting eigenvalue distributions
$$
\mathcal{D}_{m^{-(l+1)/2}} (\sigma_s(\tfrac{2}{l+1}) \boxplus \hdots \boxplus \sigma_s(\tfrac{2}{l+1}))
=
\sigma_s(\tfrac{2}{l+1}) \,.
$$
Here, $\myd_c$ is defined as in Equation \eqref{eq:scaling},
% Here, for $c>0$, $\mathcal{D}_c \mu$ denotes the induced measure 
% of the probability measure $\mu$ under the mapping $z \mapsto cz$, 
and the last step follows from the fact that $\sigma_s(\tfrac{2}{l+1})$  
is $\boxplus$-stable with stability index $\tfrac{2}{l+1}$.
Hence, again by ``asymptotic freeness'' (see~Proposition\@ \ref{prop:asymptoticfreeness}\,(d)),
we find that for any $\alpha \in \mathbb{C}$, the limiting eigenvalue distribution of the matrices 
$\widetilde\V_n + \J_n(\alpha)$ is given by $\sigma_s(\frac{2}{l+1}) \boxplus B(\alpha)$.
It therefore follows from Theorem \ref{thm:main} and Equation \eqref{eq:H}
that the limiting eigenvalue distribution of the~matrices $\widetilde\F_n$ is also given by $\myh(\sigma_s(\frac{2}{l+1}))$.
\end{proof}

\pagebreak[2]

\begin{proof}[Formal Proof of Theorem \ref{thm:newpower}]
Here the formal proof is almost identical to that of Theorem \ref{thm:new}.
The only difference is that in the first part of the proof (i.e.\@ when $m=1$),
we~addition\-ally use Proposition \eqref{prop:asymptoticfreeness}\,(b)
to see that the $S$-transform of the limiting eigenvalue distribution 
of the matrices $(\Y_n^{(r)})^{-l} ((\Y_n^{(r)})^{-l})^*$ is given by 
$S(z) = (-z)^{-l}$.
\end{proof}

\begin{remark}
In~principle, the limiting density $f(r)$ 
in Theorems \ref{thm:new} and \ref{thm:newpower}
can be found using Equation \eqref{eq:GKT}. % or \eqref{eq:haagerup}.
In our situation, it is easy to check that the equation for $\psi(\alpha)$ reduces to
$$
\psi(\alpha) (1-\psi(\alpha)) = |\alpha|^2 (1-\psi(\alpha))^{l+1} \,.
$$
Thus, using that $\psi(\alpha)$ is continuous with values in $[0,1]$ and $\psi(\alpha) \ne 1$
for $\alpha \approx 0$ (see Sections 6 and 7 in \cite{GKT:2014}), we obtain,
for $l=0,1,2,3,$
$$
\psi_0(r) = 1 \wedge r^2 \,,\quad
\psi_1(r) = \frac{r^2}{1+r^2} \,,\quad
\psi_2(r) = 1 - \frac{2}{\sqrt{1+4r^2} + 1} \,,
% \psi_2(r) = \frac{1+2r^2-\sqrt{1+4r^2}}{2r^2} \,,
$$
$$
\psi_3(r) = 1 - \frac{3}{(1 + v^2(r) + w^2(r))^2} \,,
% \psi_3(r) = 1 - \frac{2^{4/3} - 2^{2/3} \, 3 \, w^2(r)}{6r \, w(r)} \,,
$$
and therefore
$$
f_0(r) = \tfrac{1}{\pi} \, \pmb{1}_{(0,1)}(r) \,,\quad
f_1(r) = \frac{1}{\pi (1+r^2)^2} \,,\quad
f_2(r) = \frac{2}{\pi \sqrt{1+4r^2} (1 + 2r^2 + \sqrt{1+4r^2})} \,,
% f_2(r) = \frac{1+2r^2-\sqrt{1+4r^2}}{2\pi r^4 \sqrt{1+4r^2}} \,,
$$
$$
f_3(r) = \frac{27 (v(r) + w(r))}{\pi \sqrt{4 + 27r^2} (1 + v^2(r) + w^2(r))^3} \,,
% f_3(r) = \frac{2^{7/3} - 2^{5/3} \, 3 \, w^2(r) - 2^{7/3} \, 9r \, w^3(r) + 2^{2/3} \, 27r \, w^5(r)}{162\pi r^3 \, \sqrt{r^2 + \tfrac{4}{27}} \, w^4(r)}  \,,
$$
where we have set
$$
v(r) := \Big( \tfrac12 \sqrt{4 + 27r^2} + \tfrac12 \sqrt{27}r \Big)^{1/3}
\quad\text{and}\quad
w(r) := \Big( \tfrac12 \sqrt{4 + 27r^2} - \tfrac12 \sqrt{27}r \Big)^{1/3}
% w(r) := \Big( \sqrt{r^2 + \tfrac{4}{27}} - r \Big)^{1/3}
$$
for abbreviation.

\medskip

\begin{center}
\includegraphics[width=12cm]{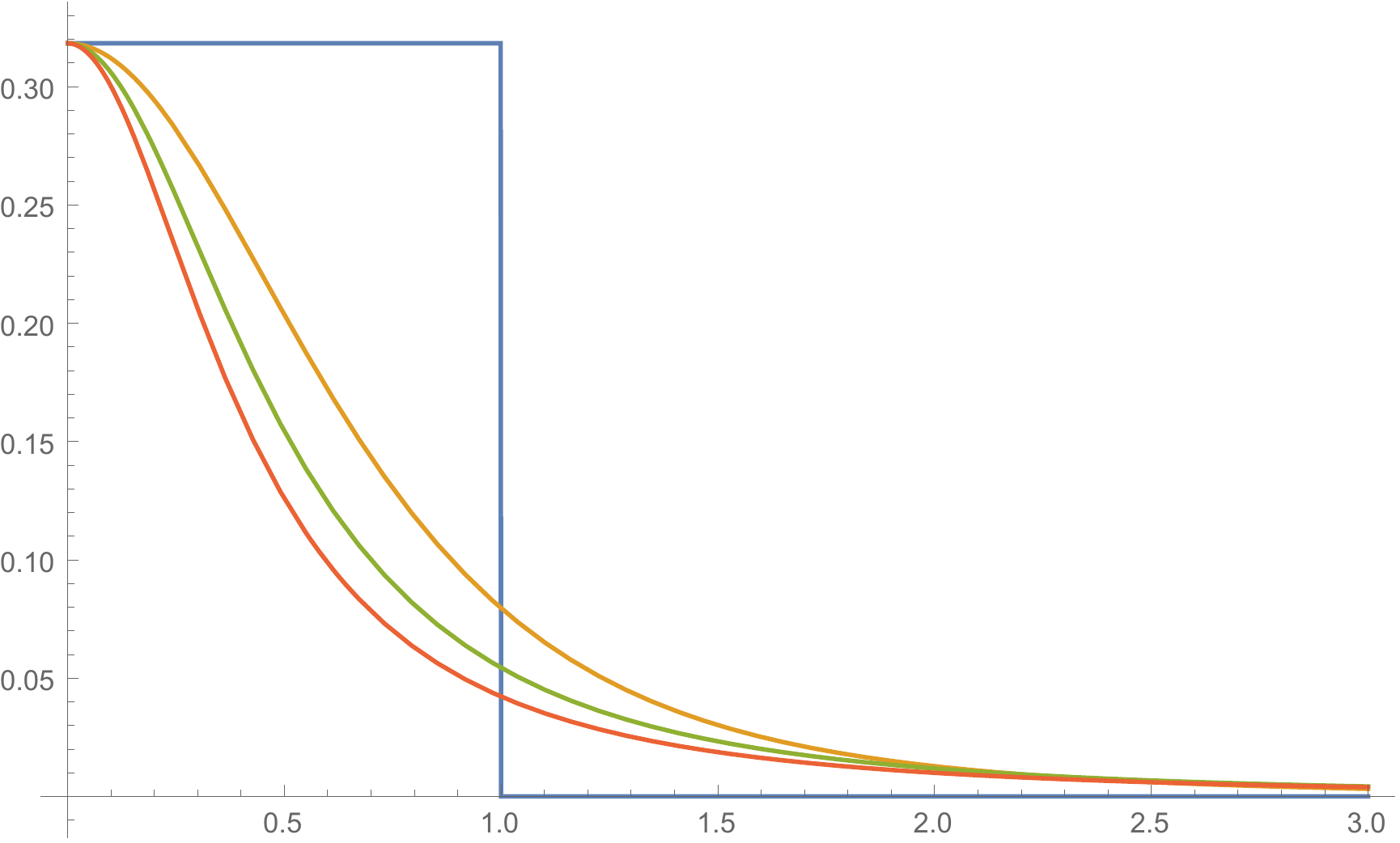}

The limiting eigenvalue densities $f_l(r)$ (viewed along a line through the~origin)
for $l=0$ (blue), $l=1$ (yellow), $l=2$ (green) and $l=3$ (red).
\end{center}

\medskip

Also, let us mention that the paper \cite{ARRS:2013} provides a stochastic representation 
of the~limiting eigenvalue modulus distribution.
\qed
\end{remark}

\pagebreak[2]

\begin{remark}
It seems natural to ask whether there exist further examples of random matrices $\F_n^{(0)}$
such that for any $m \in \mynat$, 
$\F_n^{(0)}$ and $\F_n^{(1)} + \dots + \F_n^{(m)}$ have the same limiting eigenvalue distributions
after appropriate rescaling. However, it~turns out that within the class of products 
of independent Girko-Ginibre matrices and their inverses,
there~exist no further examples beyond those mentioned in Theorem \ref{thm:new},
apart from possible permutations of the exponents $\pm 1$.
Indeed, suppose that $\F_n^{(0)}$ is a~product 
of $p$ factors $\Y_n^{(r)}$ and $q$ factors $(\Y_n^{(r)})^{-1}$
(all of them independent, and in arbitrary order), 
and let $\W_n$ and $\V_n$ be defined as in \eqref{eq:hermitization}.
Then, arguing as in the formal proof of Theorem \ref{thm:new}, we find that 
the corresponding $S$-transforms $S_{\W}$ and $S_{\V}$ are given by
$$
S_\W(z) = \frac{(-z)^{q}}{(1+z)^{p}}
\quad\text{and}\quad
S_\V(z) = \frac{\i^q \, z^{(q-1)/2}}{(1+z)^{(p-1)/2}} \,,
$$
respectively, and by Proposition \ref{prop:symmetric-free-stable},
the latter is the $S$-transform of a symmetric $\boxplus$-stable distribution
if and only if $p = 1$ and $q \in \mynat_0$. Now use the observation that
if $\mu^{\boxplus m}$ is not a rescaled version of $\mu$,
then $\myh(\mu^{\boxplus m})$ is not a rescaled version of $\myh(\mu)$.
% see also the comments below Definition \ref{def:oplus} below.
\qed
\end{remark}

% \medskip
\smallskip

\section{Free Additive Convolution on $\mycmplx$}
\label{sec:convolution}

Roughly speaking, if $\mu_1$ and $\mu_2$ are two prob\-ability~measures on the real line
and $\A_n$~and~$\B_n$ are Hermitian random matrices ``in general position''
and with limiting spectral distributions $\mu_1$ and $\mu_2$,
then the limiting spectral distribution of the sum $\A_n + \B_n$
is given by the free additive convolution $\mu_1 \boxplus \mu_2$.
It is natural to ask the analogous question for \emph{non-Hermitian} random matrices:
If $\mu_1$ and $\mu_2$ are two probability measures on the complex plane
and $\A_n$~and~$\B_n$ are non-Hermitian random matrices ``in general position''
and with limiting spectral distributions $\mu_1$~and~$\mu_2$,
does there exist a convolution $\mu_1 \oplus \mu_2$
which describes the limiting spectral distribution of the sum $\A_n + \B_n$\,?

In the sequel, we will always assume that $\A_n$ and $\B_n$ are bi-unitary invariant.
% Clearly, this implies that the limiting spectral distributions $\mu_1$ and $\mu_2$
% (if existent) are rotation-invariant. Furthermore, 
Then, in view of the results from Section~2,
it seems reasonable to expect that the limiting spectral distributions $\mu_1$ and $\mu_2$
(if existent) belong to the class $H$ introduced above Equation \eqref{eq:H}. 
It therefore seems natural to restrict the definition of the convolution $\oplus$ 
to probability measures in this class.

Hence, suppose that $\mu_1$ and $\mu_2$ are two probability measures of class $H$
and that $\A_n$ and $\B_n$ are independent bi-unitary invariant random matrices 
with limiting spectral distributions $\mu_1$ and $\mu_2$, respectively.
Also, suppose that these matrices satisfy the assumptions of Theorem \ref{thm:main}.
Then, if $\widetilde\nu_1$ and $\widetilde\nu_2$ are the limiting symmetrized 
singular value distributions of $\A_n$ and $\B_n$, we have 
$\mu_1 = \myh(\widetilde\nu_1)$ and $\mu_2 = \myh(\widetilde\nu_2)$
by Equation \eqref{eq:H}.
Furthermore, suppose that the matrix sums \linebreak $\A_n + \B_n$ have the limiting symmetrized 
singular value distribution $\widetilde\nu_1 \boxplus \widetilde\nu_2$ \linebreak
(which seems very natural in view of Proposition \ref{prop:asymptoticfreeness})
and that they also satisfy the~assumptions of Theorem \ref{thm:main}.
Then, again by Equation \eqref{eq:H}, the associated limiting eigenvalue distribution
is given by $\myh(\widetilde\nu_1 \boxplus \widetilde\nu_2)$.
This leads to the following definition:

\begin{definition}
\label{def:oplus}
Given two probability measures $\mu_1$ and $\mu_2$ of class $H$, \linebreak
set $\mu_1 \oplus \mu_2 := \myh(\myh^{-1}(\mu_1) \boxplus \myh^{-1}(\mu_2))$.
\end{definition}

\pagebreak[1]

Note that, by definition, we have 
\begin{align}
\label{eq:addition}
\myh(\widetilde\nu_1 \boxplus \widetilde\nu_2) = \myh(\widetilde\nu_1) \oplus \myh(\widetilde\nu_2)
\end{align}
for any symmetric probability measures $\widetilde\nu_1$ and $\widetilde\nu_2$ on $\myreal^*$
satisfying \eqref{eq:log-integrability-2}.

\begin{remark*}
It seems a bit unsatisfactory that the above motivation relies (inter alia)
on the assumption that the matrices $\A_n$, $\B_n$ and $\A_n + \B_n$
satisfy the Conditions $(C0)$, $(C1)$ and $(C2)$ from Theorem \ref{thm:main}.
An alternative might be to work~with matrices of the form $\U_n \T_n \V_n^*$,
where $\U_n$~and~$\V_n$ are independent unitary matrices of size $n \times n$
and $\T_n$ are deterministic diagonal matrices of dimension $n \times n$
with positive elements on the main diagonal.
The \emph{single ring theorem} \cite{GKZ:2011,GZ:2012,RV:2012} provides sufficient conditions
for the~convergence of the empirical spectral distributions of these matrices,
but it~is also subject to certain technical conditions.
Also, it is worth noting that the sums of independent unitary matrices have recently been investigated
by \textsc{Basak} and \textsc{Dembo} \cite{BD:2013}.
\end{remark*}

\begin{remark*}
The convolution $\oplus$ may also be interpreted in terms of free probability:
Given $\mu_1$ and $\mu_2$ in $H$, pick $R$-diagonal elements $x_1$ and $x_2$
(in some $W^*$-probability space) such that 
the Brown measure of $x_1$ is $\mu_1$,
the Brown measure of $x_2$ is $\mu_2$,
and $x_1$~and~$x_2$ are $*$-free.
Then $\mu_1 \oplus \mu_2$ is the Brown measure of $x_1 + x_2$,
as follows from the results in \cite{HL:2000,HS:2007}.
\end{remark*}

\pagebreak[2]

It seems natural to introduce the concept of a (strictly) $\oplus$-stable distribution.
Recall that $\myd_c$ denotes the scaling operator on the class of probability measures.

\begin{definition}
A probability measure $\mu$ of class $H$ is called \emph{$\oplus$-stable} 
if there~exists a constant $\alpha > 0$ such that
$
\mu^{\oplus m} = \mathcal{D}_{m^{1/\alpha}} \mu
$
for all $m \in \mynat$.
\end{definition}

Similarly as above, we will call the constant $\alpha$
the \emph{stability index} of the $\oplus$-stable distribution $\mu$.
Using Equations \eqref{eq:scaling} and \eqref{eq:addition}, it is easy to see that 
$\widetilde\nu$ is $\boxplus$-stable if and only if $\myh(\widetilde\nu)$ is $\oplus$-stable.
Therefore, the $\oplus$-stable distributions in $H$
are in one-to-one correspondence with the symmetric $\boxplus$-stable distributions on $\myreal^*$.

\begin{remark}
Using the $S$-transforms of the symmetric $\boxplus$-stable distributions \linebreak
(see Proposition~\ref{prop:symmetric-free-stable}),
the densities of the $\oplus$-stable distributions may be described
a bit more closely by means of either \eqref{eq:alternative} or \eqref{eq:haagerup}.
% By Theorem \ref{thm:main}, we~have the following formal result:
% Suppose that $\alpha \in (0,2)$ and $r \in (0,1)$. Then
For instance, for $\alpha \in (0,2)$,
$\psi(r)$ is given by the unique solution in the interval $(0,1)$ to the equation
$$
\frac{\psi(r)}{(1-\psi(r))^{(2/\alpha)-1}} = r^2 \,,
% \psi(r) (1-\psi(r)) = r^2 (1-\psi(r))^{2/\alpha} \,,
$$
and $f(r)$ is given by % is given by $\psi'(r)/2\pi r$.
$$
f(r) = \frac{\psi(r) (1 - \psi(r))}{\pi r^2 (1 + (\tfrac2\alpha - 2) \psi(r))} \,.
$$
% Indeed, this also follows from Equation \eqref{eq:haagerup}.
\end{remark}

\pagebreak[2]

Let us now turn to the question whether the $\oplus$-stable distributions arise
as the~limiting eigenvalue distributions of some random matrix models.
As we have already seen in Section~\ref{sec:formalproof}, 
by using products of independent Girko--Ginibre matrices and their inverses,
we only get random matrix models for $\alpha = \tfrac{2}{l+1}$, with $l \in \mynat_0$.
However, for general $\alpha \in (0,2)$, we can still take a product consisting of 
a Ginibre matrix, a diagonal matrix and a unitary matrix:

\begin{proposition}
\label{prop:matrixmodel}
Let $\alpha \in (0,2)$, let $\widetilde\alpha := 2\alpha/(2+\alpha) \in (0,1)$,
and let $\sigma_p(\widetilde\alpha)$ denote the positive $\boxplus$-stable distribution
with parameter $\widetilde\alpha$ $($see \eg Section~4 in~\cite{AP:2009}$)$.
Let $\xi$ be a random variable with distribution
$\sigma_p(\widetilde\alpha) \boxtimes B(1)$,
and let $G$ be the distribution function of $|\xi|$.
Let $\F_n = \Y_n \T_n \U_n^*$, where $\Y_n$ is a Gaussian random matrix,
$\T_n$ is a deterministic diagonal matrix with the elements $G(\tfrac{j}{n+1})$, $j=1,\hdots,n$,
on the~main~diagonal, $\U_n$ is a random unitary matrix $($with Haar distribution$)$,
and $\Y_n$ and $\U_n$ are independent.
Then $\mu(\F_n) \to \myh(\sigma_s(\alpha))$ weakly in probability.
\end{proposition}

Here and below, the free multiplicative convolution $\nu \boxtimes \mu$ of a distribution $\nu$ on~$(0,\infty)$
and a symmetric distribution $\mu$ on $\myreal^*$ is defined as in \cite{AP:2009}.
Then $\nu \boxtimes \mu$ is~again a symmetric distribution on $\myreal^*$,
and by Lemma 8 in \cite{AP:2009}, we have
\begin{align}
\label{eq:arizmendi}
\myq(\nu \boxtimes \mu) = \nu \boxtimes \myq \mu \boxtimes \nu \,.
\end{align}

\medskip

\begin{proof}[Proof of Proposition \ref{prop:matrixmodel}]
Fix $t > 0$, set $\xi(t) := (\xi \wedge t) \vee (-t)$,
and let $\T_{n}(t)$ and $\F_{n}(t)$
be defined as in the proposition, but with $\xi$ replaced by $\xi(t)$.
Additionally, let $\V_n(t)$ and $\W_n(t)$ be defined as in \eqref{eq:hermitization}.
Clearly, the~matrices $\T_n(t) \T_n(t)^*$ are deterministic
with $\mu(\T_n(t) \T_n(t)^*) \to \myl(\xi(t)^2)$ weakly.
Moreover, it is well-known that $\mu(\Y_n \Y_n^*) \to \gamma$
weakly almost surely. Since the~matrices $\Y_n$ are bi-unitary invariant,
it follows by almost~sure asymptotic freeness (see~\cite[Section 4.3]{HP:2000})
that
$$
\mu(\W_n(t)) \to \gamma \boxtimes \myl(\xi(t)^2) \ \text{weakly almost surely} \,.
$$
Also, since the~matrices $\Y_n \T_n \U_n^*$ are bi-unitary invariant,
it follows by almost~sure asymptotic freeness (see \cite[Section 4.3]{HP:2000})
that
$$
\mu(\V_n(t) + \J_n(\alpha)) \to \myq^{-1}(\gamma \boxtimes \myl(\xi(t)^2)) \boxplus B(\alpha) \ \text{weakly almost surely} \,,
$$
for any $\alpha \in \mathbb{C}$.
Now, it is easy to see that for any $\varepsilon > 0$,
there exists some $t > 0$ such that $\rank(\T_n - \T_n(t)) \leq \varepsilon n$
\pagebreak[1] and therefore $\rank(\F_n - \F_n(t)) \leq \varepsilon n$
for~all~$n \in \mynat$. This implies
$$
\mu(\W_n) \to \gamma \boxtimes \myl(\xi^2) \ \text{weakly almost surely}
$$
as well as
$$
\mu(\V_n + \J_n(\alpha)) \to \myq^{-1}(\gamma \boxtimes \myl(\xi^2)) \boxplus B(\alpha) \ \text{weakly almost surely} \,,
$$
for any $\alpha \in \mathbb{C}$.

We will show that 
\begin{align}
\label{eq:KT}
\mu_\V =  \myq^{-1}(\gamma \boxtimes \myl(\xi^2)) = \sigma_s(\alpha) \,.
\end{align}
It is well-known that the $S$-transform of $\sigma_p(\widetilde\alpha)$
is given by $z^{(1/\widetilde\alpha)-1}/b^{1/\alpha}$,
where $b$ is the same as in Proposition \ref{prop:symmetric-free-stable},
and that the $S$-transform of $B(1)$ is given by $\sqrt{\frac{z+1}{z}}$.
\linebreak[2] Thus, the $S$-transform of $\xi^2$ is given by 
$$
S_{\xi^2}(z) = \frac{z}{z+1} S_\xi^2(z) = \frac{z}{z+1} \left(\frac{z^{(1/\widetilde\alpha)-1}}{b^{1/\alpha}} \sqrt{\frac{z+1}{z}}\right)^2 = \frac{z^{(2/\widetilde\alpha)-2}}{b^{2/\alpha}} = \frac{z^{(2/\alpha)-1}}{b^{2/\alpha}}  \,,
$$
and the $S$-transform of $\mu_{\V}$ is given by
$$
S_{\V}(z) = \sqrt{\frac{z+1}{z} S_\W(z)} = \sqrt{\frac{z+1}{z} \frac{1}{z+1} \frac{z^{(2/\alpha)-1}}{b^{2/\alpha}}} = \frac{z^{(1/\alpha)-1}}{b^{1/\alpha}} \,,
$$
which proves our claim \eqref{eq:KT} by Proposition \ref{prop:symmetric-free-stable}.

Using the fact that the positive $\boxplus$-stable distribution $\sigma_p(\widetilde\alpha)$ 
has a density which vanishes in a neighborhood of the origin \cite[Theorem A.1.4]{BP:1999}
and which is of order $\myo(x^{-\widetilde\alpha-1})$ as $x \to \infty$ \cite[Theorem A.2.1]{BP:1999},
it is straightforward to check that the~matrices $\T_n \U_n^*$  satisfy Condition $C_\text{simple}$
introduced in Remark \ref{rem:C-simple} below.
Thus, it follows from~Lemma~\ref{lemma:C++} below that the~matrices 
$\F_n = \Y_n \T_n \U_n^*$ satisfy the conditions $(C0)$, $(C1)$ and $(C2)$ from Theorem~\ref{thm:main}.
We~may therefore invoke this theorem to~conclude that $\mu(\F_n) \to \myh(\sigma_s(\alpha))$
weakly in~probability.
\end{proof}

\begin{remark}
Proposition \ref{prop:matrixmodel} is closely related to an observation in \textsc{Arizmendi} and \textsc{Perez-Abreu} \cite{AP:2009}
which states that
\begin{align}
\label{eq:AP}
\sigma_s(\alpha) = \sigma_p(\widetilde\alpha) \boxtimes \sigma_{\text{Wigner}} \,,
\end{align}
where $\sigma_{s}(\alpha)$ and $\sigma_{p}(\widetilde\alpha)$ denote the symmetric and positive $\boxplus$-stable distribution
of~parameter $\alpha$ and $\widetilde\alpha$, respectively.
In fact, using \eqref{eq:arizmendi} and the random variable $\xi$ from~above, we have 
\begin{multline*}
  \myq ( \sigma_{p}(\widetilde\alpha) \boxtimes \sigma_{\text{Wigner}} )
= \sigma_{p}(\widetilde\alpha) \boxtimes \gamma \boxtimes \sigma_{p}(\widetilde\alpha) \\
= \gamma \boxtimes \Big( \sigma_{p}(\widetilde\alpha) \boxtimes \delta_1 \boxtimes \sigma_{p}(\widetilde\alpha) \Big)
= \gamma \boxtimes \myq ( \sigma_{p}(\widetilde\alpha) \boxtimes B(1) )
= \gamma \boxtimes \myl(\xi^2) \,,
\end{multline*}
which shows that the relation \eqref{eq:AP} is equivalent to the relation \eqref{eq:KT}
checked in the preceding proof.
\qed
\end{remark}

\medskip

\section{A General Limit Theorem for Sums of Products \\ of Independent Random Matrices}
\label{sec:limit-theorem}

\subsection{Overview}
\label{sub:overview}

In this section we prove a general result (see Theorem~\ref{thm:limit-theorem} below)
about the limiting singular value and eigenvalue distributions of sums of products 
of independent Girko--Ginibre matrices and their inverses. In particular, 
this result contains Proposition \ref{prop:simple-limit-theorem} from the introduction,
and it allows for a rigorous proof of Theorem \ref{thm:new}. \pagebreak[2]
To derive Theorem \ref{thm:limit-theorem},
we apply the general framework from \cite{GKT:2014}.
In~Subsection \ref{sub:framework}, we~summarize the technical conditions
and the main universality results from \cite{GKT:2014} to make the presentation self-contained.
In Subsection \ref{sub:main-result}, we~state Theorem \ref{thm:limit-theorem}.
Subsections \ref{sub:condition-A} -- \ref{sub:condition-C}
prepare for the proof of Theorem~\ref{thm:limit-theorem} 
by verifying the technical conditions from \cite{GKT:2014}.
Subsection \ref{sub:main-proof} contains the proof of Theorem \ref{thm:limit-theorem}.
Finally, in Subsection \ref{sub:rigor}, we sketch the rigorous proof
of Theorems \ref{thm:new} and \ref{thm:newpower}.
Some relevant results from the~literature are collected in Section~\ref{sec:auxiliary-results}.

\subsection{General Framework}
\label{sub:framework}

We consider random matrices of the form
\begin{align}
\label{eq:sums-of-products-2}
\F_n := \sum_{q=1}^{m} \F_n^{(q)} := \sum_{q=1}^{m} \prod_{r=1}^{l} (\X_n^{((q-1)l+r)})^{\varepsilon_{r}} \,,
\end{align}
where $m \in \mynat$, $l \in \mynat$ and $\varepsilon_1,\hdots,\varepsilon_l \in \{ -1,+1 \}$ are fixed
and the~$\X_n^{(q)}$ are in\-dependent Girko--Ginibre matrices as in the introduction.
% 
% We consider random matrices $\F_n$ as in \eqref{eq:sums-of-products},
% \ie sums of products of independent random matrices with independent matrix entries.
%
A~major~step in \cite{GKT:2014} is to~prove the \emph{universality}
of the limiting singular value and eigenvalue distributions, i.e.\@ to~show 
that these distributions (if existent) do not depend on the distributions 
of the matrix entries apart from a few moment conditions
as in \eqref{eq:real-moments} -- \eqref{eq:uniform-integrability}.
To~state this more precisely, we need two sets of random matrices.

To this end, it seems convenient to view $\mathbf{F}_n$ as a \emph{matrix function}
(by slight abuse of notation) and to write
\begin{align}
\label{eq:matrixfunction}
\mathbf{F}_n(\Z_n^{(1)},\hdots,\Z_n^{(ml)}) := \sum_{q=1}^{m} \prod_{r=1}^{l} (\Z_n^{((q-1)l+r)})^{\varepsilon_{r}} \,,
\end{align}
where $m \in \mynat$, $l \in \mynat$ and $\varepsilon_1,\hdots,\varepsilon_l \in \{ -1,+1 \}$
are the same as in \eqref{eq:sums-of-products}
and $\Z_n^{(q)} = (Z^{(q)}_{jk})_{jk=1,\hdots,n}$
is a matrix in the indeterminates $Z^{(q)}_{jk}$, $q=1,\hdots,ml$.
Then, we may write $\F_n(\X) := \F_n(\X_n^{(1)},\dots,\X_n^{(ml)})$ 
for the random matrices built from the random matrices
$\X_{n}^{(q)} := (\tfrac{1}{\sqrt{n}} X_{jk}^{(q)})_{j,k=1,\hdots,n}$
and $\F_n(\Y) := \F_n(\Y_n^{(1)},\dots,\Y_n^{(ml)})$ 
for the corresponding random matrices built from the \emph{Gaussian} random matrices
$\Y_{n}^{(q)} := (\tfrac{1}{\sqrt{n}} Y_{jk}^{(q)})_{j,k=1,\hdots,n}$.
We always assume that the families 
$(X_{jk}^{(q)})_{j,k,q \in \mynat}$ and $(Y_{jk}^{(q)})_{j,k,q \in \mynat}$
are defined on the same probability space and independent.
When the choice of the matrices $\X_n^{(1)},\hdots,\X_n^{(ml)}$ is clear from the context, 
we~also write $\mathbf{F}_n$ instead of $\mathbf{F}_n(\X)$.

\begin{remark*}
More generally, using the arguments from this section, we might deal with matrix functions of the form
$$
\F_n := \sum_{q=1}^{m} \F_n^{(q)} := \sum_{q=1}^{m} \prod_{r=1}^{l_q} (\Z_n^{(i_{q,r})})^{\varepsilon_{q,r}} \,,
$$
where $m,l_1,\hdots,l_m \in \mynat$, $\varepsilon_{q,r} \in \{ +1,-1 \}$, % $i_{q,r} \in \mynat$,
the~indices $i_{q,r} \in \mynat$ are pairwise different, and all~parameters do not depend on $n$.
That is to say, the numbers and the types of the factors in the $m$ summands need not be the same.
\end{remark*}

In our investigation of the limiting spectral distributions of the matrices $\F_n$,
we will also consider the \emph{shifted matrices} $\mathbf{F}_n - \alpha \I_n$, with $\alpha \in \mathbb{C}$,
the \emph{regularized matrices} $\mathbf{F}_{n,t}$, with $t > 0$,
and their combinations $\mathbf{F}_{n,t} - \alpha \I_n$.
Here, the regularized matrices $\mathbf{F}_{n,t}$ arise from the regularized matrix functions
\begin{align}
\label{eq:regularizedmatrixfunction}
\mathbf{F}_{n,t}(\Z_n^{(1)},\hdots,\Z_n^{(ml)}) := \sum_{q=1}^{m} \prod_{r=1}^{l} (\Z_n^{((q-1)l+r)})_t^{\varepsilon_{r}} \,,
\end{align}
where $(\Z_n)_t^\varepsilon := \Z_n$ for $\varepsilon = +1$ 
and $(\Z_n)_t^\varepsilon := (\Z_n^* \Z_n + t \I_n)^{-1} \Z_n^*$ for $\varepsilon = -1$.
Note~that, by definition, the regularization has no effect when $\varepsilon = +1$
and~that $\lim_{t \downarrow 0} (\Z_n)^{-1}_t = (\Z_n)^{-1}$ when $\Z_n$ is invertible.

Furthermore, fix a sequence $(\tau_n)_{n \in \mynat}$ of positive real numbers
such that $\tau_n \to 0$ and $\tau_n \sqrt{n} \to \infty$,
and set
$$
\widehat{X}_{jk}^{(q)} := X_{jk}^{(q)} \, \pmb{1}_{\{ |X_{jk}^{(q)}| \leq \tau_n \sqrt{n} \}} \,,\quad
% \quad\text{and}\quad
\widehat{Y}_{jk}^{(q)} := Y_{jk}^{(q)} \, \pmb{1}_{\{ |Y_{jk}^{(q)}| \leq \tau_n \sqrt{n} \}} \qquad
(j,k,q \in \mynat),
$$
$$
Z_{jk}^{(q)}(\varphi) := (\cos \varphi) \, \widehat{X}_{jk}^{(q)} + (\sin \varphi) \, \widehat{Y}_{jk}^{(q)} \qquad
(j,k,q \in \mynat; \, 0 \leq \varphi \leq \tfrac{\pi}{2}),
$$
and for $n \in \mynat$ and $0 \leq \varphi \leq \tfrac{\pi}{2}$,
set $\Z_{n}^{(q)}(\varphi) := (\tfrac{1}{\sqrt{n}} Z_{jk}^{(q)}(\varphi))_{j,k=1,\hdots,n}$
($q=1,\hdots,ml$),
$\F_n(\varphi) := \F_n(\Z_n^{(1)}(\varphi),\hdots,\Z_n^{(ml)}(\varphi))$.
Note that $\F_n(0) = \F_n(\widehat\X)$, $\F_n(\tfrac{\pi}{2}) = \F_n(\widehat\Y)$,
where $\F_n(\widehat\X)$ and $\F_n(\widehat\Y)$ are defined analogously to $\F_n(\X)$ and $\F_n(\Y)$.

\pagebreak[2]

For $n \in \mynat$, $0 \leq \varphi \leq \frac{\pi}{2}$ and $z \in \mathbb{C}^+$, introduce the Hermitian matrix
$$
\V_n(\varphi) := \left[ \begin{array}{cc} \mathbf{O} & \F_n(\varphi) \\ \F_n(\varphi)^* & \mathbf{O} \end{array} \right] \,,
$$
and the traces
\begin{align*}
g_{jk}^{(q)} &:= \trace \bigg( \frac{\partial \V_n(\varphi)}{\partial \re Z_{jk}^{(q)}} ( \V_n(\varphi) - z \I_{2n} )^{-2} \bigg) \,,
\\
\widehat{g}_{jk}^{(q)} &:= \trace \bigg( \frac{\partial \V_n(\varphi)}{\partial \im Z_{jk}^{(q)}} ( \V_n(\varphi) - z \I_{2n} )^{-2} \bigg) \,.
\end{align*}
The dependence on $n \in \mynat$, $0 \leq \varphi \leq \frac{\pi}{2}$ and $z \in \mathbb{C}^+$ is implicit here.
Also, when taking partial derivatives, we view $\mathbf{F}_n(\Z_n^{(1)},\hdots,\Z_n^{(ml)})$ 
as a function of the indeterminates $Z_{jk}^{(q)}$ (the~elements of the matrices $\Z_n^{(q)}$).
The same convention applies to partial derivatives such as
${\partial g_{jk}^{(q)}}/{\partial \re Z_{jk}^{(q)}}$,
${\partial \widehat{g}_{jk}^{(q)}}/{\partial \re Z_{jk}^{(q)}}$
etc.
Finally, for $0 \leq \theta \leq 1$, \linebreak[2] let 
$g_{jk}^{(q)}(\theta)$,
$\widehat{g}_{jk}^{(q)}(\theta)$,
${\partial g_{jk}^{(q)}}(\theta)/{\partial \re Z_{jk}^{(q)}}$,
${\partial \widehat{g}_{jk}^{(q)}}(\theta)/{\partial \re Z_{jk}^{(q)}}$,
etc.
denote the functions obtained from 
$g_{jk}^{(q)}$,
$\widehat{g}_{jk}^{(q)}$,
${\partial g_{jk}^{(q)}}/{\partial \re Z_{jk}^{(q)}}$,
${\partial \widehat{g}_{jk}^{(q)}}/{\partial \re Z_{jk}^{(q)}}$,
etc.
by replacing $Z_{jk}^{(q)}(\varphi)$ with $\theta Z_{jk}^{(q)}(\varphi)$.

\pagebreak[2]

Given a sequence of random matrices $\F_{n}$ as in \eqref{eq:matrixfunction} and a constant $t > 0$,
we~denote by $\F_{n,t}$ the associated regularized random matrices
as in \eqref{eq:regularizedmatrixfunction}.
With this notation, we~have to check the following Conditions A, B and C:

\def\supnorm#1{\Big\|#1\Big\|_{\infty}}
\def\condexp#1{\E\Big\{#1\Big|X_{jk}^{(q)}, Y_{jk}^{(q)}\Big\}}
\def\g{g_{jk}^{(q)}(\theta)}
\def\gx{\frac{\partial g_{jk}^{(q)}(\theta)}{\partial {\re Z_{jk}^{(q)}}}}
\def\gy{\frac{\partial g_{jk}^{(q)}(\theta)}{\partial {\im Z_{jk}^{(q)}}}}
\def\gxx{\frac{\partial^2 g_{jk}^{(q)}(\theta)}{\partial (\re Z_{jk}^{(q)})^2}}
\def\gyy{\frac{\partial^2 g_{jk}^{(q)}(\theta)}{\partial (\im Z_{jk}^{(q)})^2}}
\def\gxy{\frac{\partial^2 g_{jk}^{(q)}(\theta)}{\partial {\re Z_{jk}^{(q)}} \, \partial {\im Z_{jk}^{(q)}}}}
\def\h{\widehat{g}_{jk}^{(q)}(\theta)}
\def\hx{\frac{\partial \widehat{g}_{jk}^{(q)}(\theta)}{\partial {\re Z_{jk}^{(q)}}}}
\def\hy{\frac{\partial \widehat{g}_{jk}^{(q)}(\theta)}{\partial {\re Z_{jk}^{(q)}}}}
\def\hxx{\frac{\partial^2 \widehat{g}_{jk}^{(q)}(\theta)}{\partial (\re Z_{jk}^{(q)})^2}}
\def\hyy{\frac{\partial^2 \widehat{g}_{jk}^{(q)}(\theta)}{\partial (\im Z_{jk}^{(q)})^2}}
\def\hxy{\frac{\partial^2 \widehat{g}_{jk}^{(q)}(\theta)}{\partial {\re Z_{jk}^{(q)}} \, \partial {\im Z_{jk}^{(q)}}}}

\pagebreak[2]
\smallskip

\textbf{Condition A:} \\
For $\F_n = \F_n(\X)$ and $\F_n = \F_n(\Y)$,
the matrices $\F_n$ satisfy the~following con\-dition:

For each $\alpha \in \mathbb{C}$ and $z \in \mathbb{C}^{+}$,
we have $\lim_{t \to 0} \limsup_{n \to \infty} |s_{n,t}(z) - s_n(z)| = 0$ in~probability,
where $s_n(z)$ and $s_{n,t}(z)$ are the Stieltjes transforms of the Hermitian matrices 
$(\F_n - \alpha \I_n)(\F_n  - \alpha \I_n)^*$ and $(\F_{n,t}  - \alpha \I_n)(\F_{n,t} - \alpha \I_n)^*$, 
respectively.

\pagebreak[2]
\smallskip

\textbf{Condition B:} \\
For each $t > 0$, $\alpha \in \mathbb{C}$ and $z \in \mycmplx^+$, 
the functions $\g$ $(0 \le \theta \le 1)$ associated with 
the matrix functions $\mathbf{F}_{n,t} - \alpha \I_n$ 
satisfy the following bounds:
\begin{align}
&\sup_{j,k,q} \supnorm{\condexp{\g}} \le A_0
\label{B0}\tag{B0} \displaybreak[2]\\
% &\sup_{j,k,q} \supnorm{\condexp{\h}} \le A_0
% \label{B0}\tag{B0}\\
&\sup_{j,k,q} \max\Big\{ \supnorm{\condexp{\gx}},\supnorm{\condexp{\gy}} \Big\} \le A_1
\label{B1}\tag{B1} \displaybreak[2]\\
% &\sup_{j,k,q} \max\Big\{ \supnorm{\condexp{\hx}},\supnorm{\condexp{\hy}} \Big\} \le A_1
% \label{B1}\tag{B1}\\
&\sup_{j,k,q} \max\Big\{ \supnorm{\condexp{\gxx}},\supnorm{\condexp{\gyy}},\notag\\
&\qquad\qquad\qquad\qquad\qquad\qquad \supnorm{\condexp{\gxy}} \Big\} \le A_2
\label{B2}\tag{B2}
% &\sup_{j,k,q} \max\Big\{ \supnorm{\condexp{\hxx}},\supnorm{\condexp{\hyy}},\notag\\
% &\qquad\qquad\qquad\qquad\qquad\qquad\qquad\supnorm{\condexp{\hxy}} \Big\} \le A_2
% \label{B2}\tag{B2}\\
\notag
\end{align}
Here $A_0,A_1,A_2$ are certain constants which may depend on $t > 0$, $\alpha \in \mycmplx$ and $z \in \mycmplx^+$
but not on $n \in \mynat$, $\varphi \in [0,\tfrac\pi2]$ or $\theta \in [0,1]$.
Furthermore, similar bounds hold for the~functions $\widehat{g}_{jk}^{(q)}(\theta)$ $(0 \le \theta \le 1)$
and their partial derivatives. 

\pagebreak[3]
\smallskip

\textbf{Condition C:} \\
For $\F_n = \F_n(\X)$ and $\F_n = \F_n(\Y)$,
the matrices $\F_n$ satisfy the~following con\-ditions:

($C0$) There exists some $p>0$ such that 
\begin{align*}
\frac1n\sum_{k=1}^ns_k^p(\F_{n})
\end{align*}
is bounded in probability as $n \to \infty$.

($C1$) For any fixed $\alpha \in \mathbb{C}$, 
there exists some $Q>0$ such that
\begin{align*}
\lim_{n\to\infty}\pp\Big(s_n(\F_{n}-\alpha \I_n)\le n^{-Q}\Big)=0.
\end{align*}

($C2$) For any fixed $\alpha \in \mathbb{C}$, 
there exists some $0 < \gamma < 1$ such that
for \emph{any} sequence $(\delta_n)_{n \in \mynat}$ with $\delta_n \to 0$,
\begin{equation*}
\lim_{n\to\infty}\pp\Big(\frac1n\sum_{{n_1}\le j\le{ n_2} }| \log s_j(\F_{n}-\alpha \I_n)|>\varepsilon\Big)=0
\quad\text{for all $\varepsilon > 0$},
\end{equation*}
where $n_1=[n-n\delta_n]+1$ and $n_2=[n-n^{\gamma}]$.

\pagebreak[1]
\medskip

\def\CS{C$_\text{simple}$\xspace}

\begin{remark}[Condition \CS]
\label{rem:C-simple}
It will be convenient to consider Condition~C
for more general matrices $\F_n$ than in \eqref{eq:sums-of-products-2}.
Thus, if a sequence of random matrices $\F_n$ (with $\F_n$ of dimension $n \times n$)
satisfies Conditions ($C0$), ($C1$) and ($C2$), 
we say that the~matrices $\F_n$ satisfy \emph{Condition~C}.
Also, if a sequence of random matrices $\F_n$ (with $\F_n$ of dimension $n \times n$)
satisfies Conditions ($C0$) as well as Conditions ($C1$) and ($C2$) with $\alpha = 0$,
we say that the~matrices $\F_n$ satisfy \emph{Condition~\CS}.
\end{remark}

\pagebreak[1]
\medskip

The following result is essentially contained in \cite{GKT:2014}:

\begin{theorem}[Universality of Singular Value and Eigenvalue Distributions]
\label{thm:universality} \ \newline
Let $\F_n(\X)$, $\F_n(\Y)$ be defined as above,
and let $\nu_n(\X)$, $\nu_n(\Y)$ and $\mu_n(\X)$, $\mu_n(\Y)$
denote the associated singular value and eigenvalue distributions, respectively.
\begin{enumerate}[(a)]
\item
If Conditions A and B hold, we have 
$$
\nu_n(\X) - \nu_n(\Y) \to 0 \qquad \text{weakly in probability.}
$$
\item
If Conditions A, B and C hold, we have
$$
\mu_n(\X) - \mu_n(\Y) \to 0 \qquad \text{weakly in probability.}
$$
\end{enumerate}
\end{theorem}

\pagebreak[2]
\medskip

\begin{proof} \

(a) Set $\alpha := 0$.
For $\Z = \X$ and $\Z = \Y$, let $m_{n}(z;\Z)$ and $s_n(z;\Z)$ denote the~Stieltjes transforms
of the Hermitian matrices
$$
\V_n(\Z) := \left[ \begin{array}{cc} \mathbf{O} & \F_n(\Z) \\ \F_n^*(\Z) & \mathbf{O} \end{array} \right]
\quad\text{and}\quad
\W_n(\Z) := \F_n(\Z) \F_n^*(\Z) \,,
$$
and let $m_{n,t}(z;\Z)$ and $s_{n,t}(z;\Z)$ denote the corresponding Stieltjes transforms
when $\F_n(\Z)$ is replaced with $\F_{n,t}(\Z)$. \pagebreak[2]
% Similarly, let $\nu_{n,t}(\Z)$ denote the singular value distribution
% of the matrix $\F_{n,t}(\Z)$.
Fix $t > 0$. By Condition~B and Theorem 3.2 in~\cite{GKT:2014},
we have, for each $z \in \mycmplx^{+}$,
$m_{n,t}(z;\X) - m_{n,t}(z;\Y) \to 0$ in probability 
and therefore
$s_{n,t}(z;\X) - s_{n,t}(z;\Y) \to 0$ in probability.
It therefore follows from Condition~A that, for each $z \in \mycmplx^{+}$,
$s_{n}(z;\X) - s_{n}(z;\Y) \to 0$ in probability,
which implies the~claim.

\pagebreak[2]

(b) By the same argument as in (a), the conclusion of (a) holds
not only for the singular value distributions of the matrices $\F_n$,
but also for the singular value distributions of the shifted matrices $\F_n - \alpha \I_n$,
for any fixed $\alpha \in \mathbb{C}$.
Thus, the claim follows from Condition~C and Remark~4.2 in~\cite{GKT:2014}.
\end{proof}

\pagebreak[2]

\begin{remark}
As follows from the proof, if one is only interested in the limiting singular value distributions
of the matrices $\F_n$, it suffices to assume that Conditions A and B hold with $\alpha = 0$.
\end{remark}

\pagebreak[2]

\subsection{A General Limit Theorem}
\label{sub:main-result}

We will use Theorem \ref{thm:universality} to establish the following limit theorem,
which contains Proposition \ref{prop:simple-limit-theorem} from the introduction.

\begin{theorem}
\label{thm:limit-theorem} 
Let the matrices $\F_n(\X)$ be defined as in \eqref{eq:sums-of-products}.
Then there exist non-random probability measures $\nu$ and $\mu$ on $(0,\infty)$ and $\mycmplx$,
respectively, such that
$$
\lim_{n \to \infty} \nu_n(\F_n(\X)) = \nu \qquad \text{weakly in probability}
$$
and
$$
\lim_{n \to \infty} \mu_n(\F_n(\X)) = \mu \qquad \text{weakly in probability}
$$
and the limiting distributions are the same as those for the matrices $\F_n(\Y)$
derived from Gaussian random matrices. 
More precisely, the measure $\nu$ is given by
$$
\mys \nu = \Big( \myq^{-1} ( \gamma^{\varepsilon_1} \boxtimes \cdots \boxtimes \gamma^{\varepsilon_l} )  \Big)^{\boxplus m} \,,
$$
with $\mys$ and $\myq$ as in \eqref{eq:S-and-Q},
and the measure $\mu$ is the unique prob\-ability measure on~$\mycmplx$
satisfying \eqref{eq:logpotential}, with $\mu_\V$ replaced by $\mys \nu$
$($or, equivalently, $\mu = \myh(\mys \nu)$, with $\myh$ as in \eqref{eq:H}$)$.
\end{theorem}

The proof of Theorem \ref{thm:limit-theorem} will be given below in Subsection \ref{sub:main-proof}.
The main idea is that the result is true for the~matrices $\X^{(q)}$ and their~inverses $(\X^{(q)})^{-1}$
(see Section~\ref{sec:RMT})
and that the corresponding result for the matrices $\F_n(\X)$ follows from this by means of 
induction on $l$ and on $m$.
To establish the existence of the limiting distributions in the Gaussian case,
\pagebreak[1] we will use tools from free prob\-ability.
To extend this existence to the general case and to establish universality,
we will use Theorem \ref{thm:universality}, of course,
which requires us to verify Conditions A, B and~C. For~this~purpose, 
we provide some auxiliary results in the next 3 subsections.

\subsection{On Condition A}
\label{sub:condition-A}

Let $\mathbf{F}_n = \mathbf{F}_n(\X)$ be defined as in \eqref{eq:sums-of-products}.
In~order to~obtain a matrix function which is smooth in the matrix entries
(as needed for Condition~B), we replace all inverses $(\X^{(q)})^{-1}$
with regularized inverses $(\X^{(q)})_{t}^{-1}$.
We do this in a step~by~step fashion.
Hence, fix $t > 0$, fix an index $Q$ such that $\varepsilon_Q = -1$,
and for~all the other indices $q$ with $\varepsilon_q = -1$,
fix a choice between $(\X^{(q)})^{-1}$ and $(\X^{(q)})_{t}^{-1}$.
Then the resulting matrix $\F_n$ may be represented as
\begin{align}
\label{eq:AXBC}
\mathbf{F}_{n} = \A_n (\X_n)^{-1} \B_n + \C_n \,,
\end{align}
where $\X_n \equiv \X_n^{(Q)}$ (we omit the index $Q$ for simplicity).

\pagebreak[2]

Fix $\alpha \in \mathbb{C}$, and for $0 \leq u \leq t$, let
\begin{align}
\label{eq:F-nu}
\mathbf{F}_{n,u} = \A_n (\X_{n})^{-1}_{u} \B_n + \C_n := \A_n (\X_n^* \X_n + u)^{-1} \X_n^* \B_n + \C_n \,,
\end{align}
and
\begin{align}
\label{eq:s-nu}
s_{n,u}(z) := \tfrac1n \trace \big( (\mathbf{F}_{n,u} - \alpha \I_n) (\mathbf{F}_{n,u} - \alpha \I_n)^* - z \big)^{-1} \,.
\end{align}

\pagebreak[2]

\noindent{}Note that $\F_{n,0}$ coincides with $\F_n$ if $\X_n$ is invertible.
Then, by way of induction, it will suffice to prove the following lemma:

\begin{lemma}
\label{lemma:A}
For each $n \in \mynat$, let $\X_n = (\tfrac{1}{\sqrt{n}} X_{jk})_{jk=1,\hdots,n}$
be a matrix as~in~\eqref{eq:girko-ginibre},
where the entries $X_{jk}$ are independent random variables 
satisfying \eqref{eq:real-moments} -- \eqref{eq:uniform-integrability}.
Furthermore, for each $n \in \mynat$, let $\A_n$, $\B_n$ and $\C_n$ be random matrices 
of dimension $n \times n$ such that the singular value distributions of the random matrices 
$\B_n$ and $\C_n$ \linebreak[2] converge weakly in probability to~(non-random) probability measures on $(0,\infty)$,
and let $\mathbf{F}_{n,u}$ and $s_{n,u}(z)$ be defined as in \eqref{eq:F-nu} and \eqref{eq:s-nu}.
Then, for any $z = u + iv \in \mathbb{C}^+$, we have
\begin{align} 
\label{eq:reg-0}
\lim_{t \to 0} \limsup_{n \to \infty} | s_{n,t}(z) - s_{n,0}(z) | = 0 \quad \text{in probability} \,.
\end{align}
\end{lemma}

\begin{remark}
\label{remark:no-independence}
Let us emphasize that although the matrices $\A_n$, $\B_n$, $\C_n$ and $\X_n$ in the decomposition \eqref{eq:AXBC} are independent,
this is not required in Lemma \ref{lemma:A}. 
\end{remark}

\begin{remark}
\label{remark:reg}
Lemma 8.16 in \cite{GKT:2014} contains a similar result for the case where \linebreak $\C_n = 0$.
This result is based on the additional assumption 
that the matrices $\B_n$ satisfy Condition C,
but as we shall see below, this assumption is not needed.
The main difference in the proof of Lemma \ref{lemma:A} 
(as compared to that of Lemma 8.16 in~\cite{GKT:2014}) 
is that we control the necessary \emph{auxiliary} modifications 
of the matrices $\B_n$ and $\C_n$
via the matrix rank, and not via the resolvent.
\end{remark}

\begin{proof}[Proof of Lemma \ref{lemma:A}]
For the sake of simplicity, we consider only the case $\alpha = 0$ here,
the extension to the case $\alpha \ne 0$ being straightforward.
We have to show that for any given $\varepsilon > 0$ and $\delta > 0$,
\begin{align} 
\label{eq:reg-1}
\limsup_{t \to 0} \limsup_{n \to \infty} \pp \big( | s_{n,t}(z) - s_{n,0}(z) | > \varepsilon \big) < \delta  \,.
\end{align}
Hence, fix $\varepsilon > 0$ and $\delta > 0$.
Similarly as in the proof of Lemma 8.16 in \cite{GKT:2014},
we~introduce \emph{auxiliary} modifications
of the matrices $\B_n$ and $\C_n$ before we do the~regularization
of the inverse matrices $\X_n^{-1}$.

\pagebreak[2]

For an $n \times n$ matrix $\M$,
let $s_1(\M) \geq \cdots \geq s_n(\M)$ denote the singular values.
Since the singular value distributions of $\B_n$ and $\C_n$
converge weakly in probability to (non-random) probability measures
on $(0,\infty)$, we may find $K > 1$ and $N \in \mynat$ such that
for $n \geq N$, we have
$$
\pp \left( \tfrac1n \sum_{k=1}^{n} \pmb{1}_{\{ s_k(\B_n) < K^{-1} \ \vee \ s_k(\B_n) > K \}} > \tfrac{\varepsilon v}{24} 
	  \ \vee \ \tfrac1n \sum_{k=1}^{n} \pmb{1}_{\{ s_k(\C_n) > K \}} > \tfrac{\varepsilon v}{24} \right) < \tfrac{\delta}{2} \,.
$$
Then, the modifications $\widetilde\B_n$ and $\widetilde\C_n$ are defined as follows:
For the matrix $\C_n$, take the singular value decomposition $\C_n = \U \DDelta \V^*$,
let $\widetilde\DDelta$ be the diagonal matrix obtained from $\DDelta$
by replacing the diagonal elements $\DDelta_{kk}$ with $\widetilde\DDelta_{kk} := \DDelta_{kk} \wedge K$,
and set $\widetilde{\C}_n := \U \widetilde\DDelta \V^*$. \pagebreak[2]
For the matrix $\B_n$, take the singular value decomposition $\B_n = \U \DDelta \V^*$,
let $\widetilde\DDelta$ be the diagonal matrix obtained from $\DDelta$
by replacing the diagonal elements $\DDelta_{kk}$ with $\widetilde\DDelta_{kk} := (\DDelta_{kk} \wedge K) \vee K^{-1}$,
and set $\widetilde{\B}_n := \U \widetilde\DDelta \V^*$.
Then we have 
\begin{align}
\label{eq:reg-11}
\| \widetilde\B_n \| \leq K \,,
\quad
\| \widetilde\B_n^{-1} \| \leq K \,,
\quad
\| \widetilde\C_n \| \leq K \,,
\end{align}
and for $n \geq N$, with a probability of at least $1 - \tfrac{\delta}{2}$, we also have
\begin{align}
\label{eq:reg-12}
\tfrac1n \rank(\B_n - \widetilde\B_n) \leq \tfrac{\varepsilon v}{24} \,,
\quad
\tfrac1n \rank(\C_n - \widetilde\C_n) \leq \tfrac{\varepsilon v}{24} \,.
\end{align}
Furthermore, let $\widetilde{F}_{n,u}$ and $\widetilde{s}_{n,u}(z)$ 
be defined as in \eqref{eq:F-nu} and \eqref{eq:s-nu},
but with $\B_n$ and $\C_n$ replaced by $\widetilde{\B}_n$ and $\widetilde{\C}_n$.
It then follows from \eqref{eq:reg-12} that for $n \geq N$,
with a~probability of at least $1 - \tfrac{\delta}{2}$,
we have
$$
\tfrac1n \rank(\F_{n,u} \F_{n,u}^* - \widetilde\F_{n,u} \widetilde\F_{n,u}^*) \leq \tfrac{\varepsilon v}{6}
$$
and therefore, by the rank inequality,
$$
|s_{n,u}(z) - \widetilde{s}_{n,u}(z)| \leq \tfrac{\varepsilon}{3} \,.
$$

Thus, we have reduced the proof of \eqref{eq:reg-1} to showing that
\begin{align} 
\label{eq:reg-2}
\lim_{t \to 0} \limsup_{n \to \infty} | \widetilde{s}_{n,t}(z) - \widetilde{s}_{n,0}(z) | = 0 \quad \text{in probability} \,.
\end{align}
Since we only deal with the modified matrices for the rest of the proof, 
we~omit the tildes and write $\B_n,\C_n,\F_{n,u}$ and $s_{n,u}(z)$ 
instead of $\widetilde\B_n,\widetilde\C_n,\widetilde\F_{n,u}$ and $\widetilde s_{n,u}(z)$,
respectively. Moreover, for brevity, we usually omit the index $n$.

To establish \eqref{eq:reg-2}, we may proceed similarly as in the proof of Lemma 8.16 in~\cite{GKT:2014}.
Set $\R_{u} := (\F_{u}^{} \F_{u}^* - z \I)^{-1}$, $0 \leq u \leq t$, 
and note that we have the estimates
\begin{multline}
\label{eq:reg-22}
\|\R_u\|_2 \leq v^{-1} \,,\
\|\F_u^* \R_u \F_u\|_2 \leq 1 + |z|v^{-1} \,, \\
\| \R_u \F_u \| \leq \left( v^{-1}(1+|z|v^{-1}) \right)^{1/2} \,,
\| \F_u^* \R_u \| \leq \left( v^{-1}(1+|z|v^{-1}) \right)^{1/2}
\end{multline}
as well as the representation
\begin{align}
  \R_t - \R_0
= \int_0^t \frac{d\R_u}{du} \, du 
= - \int_0^t \R_u \frac{d(\F_u^{} \F_u^*)}{du} \R_u \, du \,.
\label{eq:reg-3}
\end{align}
It is easy to check that
\begin{align*}
   \frac{d(\F_u \F_u^*)}{du}
&= \A \X^{-1}_u (\X \X^* + u\I)^{-1} \B \F_u^* + \F_u^{} \B^* (\X \X^* + u\I)^{-1} (\X^{-1}_u)^* \A^* \\
&= \A \X^{-1}_u \B \B^{-1} (\X \X^* + u\I)^{-1} \B \F_u^* \\ 
	&\qquad \,+\, \F_u^{} \B^* (\X \X^* + u\I)^{-1} (\B^*)^{-1}  \B^* (\X^{-1}_u)^* \A^* \\
&= \F_u \B^{-1} (\X \X^* + u\I)^{-1} \B \F_u^* - \C \B^{-1} (\X \X^* + u\I)^{-1} \B \F_u^* \\
	&\qquad \,+\, \F_u^{} \B^* (\X \X^* + u\I)^{-1} (\B^*)^{-1} \F_u^* - \F_u^{} \B^* (\X \X^* + u\I)^{-1} (\B^*)^{-1} \C^* \,.
\end{align*}
Thus, it follows from \eqref{eq:reg-3} that
\begin{align*}
     |\tfrac1n \trace (\R_t - \R_0)|
\leq& \int_0^t |\tfrac1n \trace ( \R_u \F_u \B^{-1} (\X \X^* + u\I)^{-1} \B \F_u^* \R_u)| \, du \nonumber\\
	&+ \int_0^t |\tfrac1n \trace ( \R_u \C \B^{-1} (\X \X^* + u\I)^{-1} \B \F_u^* \R_u)| \, du \nonumber\\
	&+ \int_0^t |\tfrac1n \trace ( \R_u \F_u^{} \B^* (\X \X^* + u\I)^{-1} (\B^*)^{-1} \F_u^* \R_u)| \, du \nonumber\\
	&+ \int_0^t |\tfrac1n \trace ( \R_u \F_u^{} \B^* (\X \X^* + u\I)^{-1} (\B^*)^{-1} \C^* \R_u)| \, du \,.
% \label{eq:reg-4}
\end{align*}
Using the inequality $|\trace(\M_1 \M_2 \M_3)| \leq \| \M_1 \| \| \M_3 \| \trace (\M_2)$
(which holds for~any $n \times n$ matrices $\M_1$, $\M_2$, $\M_3$ such that $\M_2$ is positive definite)
as~well~as \eqref{eq:reg-11} and \eqref{eq:reg-22}, we therefore obtain
\begin{align}
\label{eq:reg-5}
     \left| \tfrac1n \trace (\R_t - \R ) \right|
\leq C(K,z) \int_0^t \tfrac1n \trace (\X \X^* + u\I)^{-1} \, du \,,
\end{align}
where $C(K,z)$ is some constant depending only on $K$ and $z$.
Thus, it remains to~show that
\begin{align}
\label{eq:reg-6}
\lim_{t \to 0} \limsup_{n \to \infty} \int_0^t \tfrac1n \trace (\X_n^{} \X_n^* + u\I_n)^{-1} \, du = 0 \quad\text{in probability} \,.
\end{align}
But this follows from the fact that the random matrices $\X_n$ satisfy Condition~C;
\linebreak see the proof of Lemma 8.14 in \cite{GKT:2014} for details.
\end{proof}

\pagebreak[2]

\subsection{On Condition B}
\label{sub:condition-B}

Here we prove the following lemma:

\begin{lemma}
\label{lemma:B}
With $\mathbf{F}_n$ defined as in Equation \eqref{eq:sums-of-products},
Assumption~B holds.
\end{lemma}

\begin{proof}
For simplicity, we consider the case of real matrices only,
and we often omit the~arguments $\varphi$ and $\theta$ in our estimates. 
Furthermore, for brevity, we will assume that $t \in (0,1)$ and $v := \im z \in (0,1)$.
(Indeed, it would be sufficient to deal with parameters from these regions.)
Finally, for reasons of symmetry, we consider only the case where $q \in \{ 1,\hdots,l \}$.

Putting
$$
(\Z_n^{(r)})_t^{\varepsilon_{r}} := \Z_n^{(r)} \quad \text{for $\varepsilon_r = +1$}
$$
and 
$$
(\Z_n^{(r)})_t^{\varepsilon_{r}} := \A_{n,t}^{(r)} (\Z_n^{(r)})^* := ((\Z_n^{(r)})^* \Z_n^{(r)} + t \I)^{-1} (\Z_n^{(r)})^* \quad \text{for $\varepsilon_r = -1$}
$$
($r \in \mynat$), we have the representation
\begin{align}
\label{eq:B11}
\mathbf{F}_{n,t} = \mathbf{F}_{n,t}(\Z_n^{(1)},\hdots,\Z_n^{(ml)}) = \sum_{p=1}^{m} \prod_{r=1}^{l} (\Z_n^{((p-1)l+r)})_t^{\varepsilon_{r}} \,.
\end{align}
Setting $\H_n^{(q)} := \prod_{r=1}^{q-1} (\Z_n^{(r)})_t^{\varepsilon_{r}}$
and $\widetilde\H_n^{(q)} := \big(\prod_{r=q+1}^{l} (\Z_n^{(r)})_t^{\varepsilon_{r}}\big)^*$,
it follows that
\begin{align}
\label{eq:B12}
\frac{\partial \mathbf{F}_{n,t}}{\partial Z_{jk}^{(q)}} = \H_n^{(q)} \, \frac{\partial (\Z_n^{(q)})_t^{\varepsilon_q}}{\partial Z_{jk}^{(q)}} \, \big(\widetilde\H_n^{(q)}\big)^* \,,
\end{align}
where
$$
\frac{\partial (\Z_n^{(q)})_t^{\varepsilon_q}}{\partial Z_{jk}^{(q)}} = \mathbf{e}_j \mathbf{e}_k^T \quad \text{for $\varepsilon_q = +1$}
$$
and
$$
\frac{\partial (\Z_n^{(q)})_t^{\varepsilon_q}}{\partial Z_{jk}^{(q)}} = 
- \A_{n,t}^{(q)} \Big( \mathbf{e}_k \mathbf{e}_j^T \Z_n^{(q)} + (\Z_n^{(q)})^* \mathbf{e}_j \mathbf{e}_k^T \Big) \A_{n,t}^{(q)} (\Z_n^{(q)})^*
+ \A_{n,t}^{(q)} \mathbf{e}_k \mathbf{e}_j^T
\quad \text{for $\varepsilon_q = -1$} \,.
$$
Thus, in both cases, we have a representation of the form
\begin{align}
\label{eq:finiterank}
\frac{\partial (\Z_n^{(q)})_t^{\varepsilon_q}}{\partial Z_{jk}^{(q)}}
=
\v_1 \widetilde\v_1^* + \hdots + \v_s \widetilde\v_s^* \,,
\end{align}
where $s \in \{ 1,3 \}$ and the $\v_i$ and the $\widetilde\v_i$ are vectors of Euclidean norm bounded by~$t^{-1/2}$.
To see this, recall that $0 < t < 1$ and use the relations
\begin{multline}
\label{eq:matrixnorms}
\|\A_{n,t}^{(q)}\|_2 \leq t^{-1} \,,\
\|\Z_n^{(q)} \A_{n,t}^{(q)} (\Z_n^{(q)})^*\|_2 \leq 1 \,,\ \\
\|\A_{n,t}^{(q)} (\Z_n^{(q)})^*\|_2 \leq t^{-1/2} \,,\
\|\Z_n^{(q)} \A_{n,t}^{(q)}\|_2 \leq t^{-1/2} \,.
\end{multline}
Let $\V$ denote the Hermitization of the matrix $\F := \F_{n,t}(\varphi)-\alpha\I_n$, i.e.\@
$$
\V  = \begin{bmatrix}\mathbf O& \F \\\,\F^* & \mathbf O\end{bmatrix}
    = \begin{bmatrix}\mathbf O& \F_{n,t}(\varphi) - \alpha \I_n \\\,\F_{n,t}^*(\varphi) - \overline\alpha \I_n & \mathbf O\end{bmatrix} \,,
$$
and for fixed $z = u + iv \in \mathbb{C}^+$ with $v \in (0,1)$,
let $\mathbf R := (\V - z \I_{2n})^{-1}$ denote the corresponding resolvent matrix.
Furthermore, given a matrix $\M$ of dimension $2n \times 2n$, \pagebreak[1]
we denote the submatrices of dimension $n \times n$ by $[\mathbf M]_{\alpha\beta}$,
$\alpha,\beta=1,2$, so~that
$$
\mathbf M = \begin{bmatrix}\, [\mathbf M]_{11} & [\mathbf M]_{12} \,\\\, [\mathbf M]_{21} & [\mathbf M]_{22} \,\end{bmatrix} \,.
$$
Then $g_{jk}^{(q)}$ is a finite sum of scalar products of the form
\begin{align}
\label{eq:B010}
\widetilde{\v}_i^* (\widetilde\H_n^{(q)})^* [\mathbf R^2]_{21} \H_n^{(q)} \v_i \quad\text{and}\quad \v_i^* (\H_n^{(q)})^* [\mathbf n^2]_{12} \widetilde\H_n^{(q)} \widetilde\v_i
\end{align}
with $\v_i,\widetilde\v_i$ as in Equation \eqref{eq:finiterank}.
Since~$\| \mathbf R \| \leq v^{-1}$, it follows that
\begin{align}
\label{eq:B020}
     |g_{jk}^{(q)}(\theta)|
\leq C v^{-2} \sum_i \| \H_n^{(q)} \v_i \|_2 \| \widetilde\H_n^{(q)} \widetilde\v_i \|_2 \,.
\end{align}
Let $\Z_n$ be a random matrix of the same form as the matrices $\Z_n^{(r)}$.
Then it is easy to see that for any $r \in \mynat$, there exists a constant $C_r \in (0,\infty)$ such~that
for any $n \in \mynat$ and any \emph{deterministic} vector $\v_n \in \myreal^n$, we have 
\begin{align}
\label{eq:directproduct}
\ee\| \Z_n \v_n \|_{2}^{2r} \leq C_r \| \v_n \|_{2}^{2r}
\end{align}
% and for any $n \in \mynat$ and any (\emph{possibly random}) vector $\v_n \in \myreal^n$, we have
as well as 
\begin{align}
\label{eq:inverseproduct}
\ee\| (\Z_n)_t^{-1} \v_n \|_{2}^{2r} \leq C_r \| \v_n \|_{2}^{2r} / t^r \,.
\end{align}
In fact, \eqref{eq:directproduct} can be proved using our moment assumptions and independence,
while \eqref{eq:inverseproduct} is an immediate consequence of the fact that $\| (\Z_n)_t^{-1} \| \le t^{-1/2}$.

Using Cauchy--Schwarz inequality, independence as well as \eqref{eq:directproduct} and \eqref{eq:inverseproduct}
(applied conditionally on $X_{jk}^{(q)},Y_{jk}^{(q)}$),
we~obtain
\begin{align}
\mskip+11mu&\mskip-11mu\left| \condexp\g \right| \nonumber\\
&\leq Cv^{-2} \sum_i \Big( \ee \Big\{ \|           \H_n^{(q)}           \v_i \|_2^2 \,\Big|\, X_{jk}^{(q)},Y_{jk}^{(q)} \Big\} \Big)^{1/2} 
                     \Big( \ee \Big\{ \| \widetilde\H_n^{(q)} \widetilde\v_i \|_2^2 \,\Big|\, X_{jk}^{(q)},Y_{jk}^{(q)} \Big\} \Big)^{1/2} \qquad \nonumber\\ 
&\leq Cv^{-2} t^{-(l-1)/2} \sum_i \Big( \ee \Big\{ \|           \v_i \|_2^2 \,\Big|\, X_{jk}^{(q)},Y_{jk}^{(q)} \Big\} \Big)^{1/2} 
                                  \Big( \ee \Big\{ \| \widetilde\v_i \|_2^2 \,\Big|\, X_{jk}^{(q)},Y_{jk}^{(q)} \Big\} \Big)^{1/2} \qquad \nonumber\\ 
&\leq Cv^{-2} t^{-(l+1)/2} \,,
\label{eq:B030}
\end{align}
and Condition \eqref{B0} is proved.

\pagebreak[2]

Furthermore, using that
$$
  \frac{\partial \mathbf R}{\partial Z^{(q)}_{jk}} 
= - \mathbf R \frac{\partial\V}{\partial Z^{(q)}_{jk}} \mathbf R \,,
$$
it is easy to see that
$$
  \frac{\partial g_{jk}^{(q)}}{\partial Z_{jk}^{(q)}}
= \trace \bigg( {-} 2 \frac{\partial \V}{\partial Z_{jk}^{(q)}} \mathbf R^2 \frac{\partial \V}{\partial Z_{jk}^{(q)}} \mathbf R  + \frac{\partial^2 \V}{\partial (Z_{jk}^{(q)})^2} \, \mathbf R^2 \bigg)
$$
and
\begin{multline*}
  \frac{\partial^2 g_{jk}^{(q)}}{\partial (Z_{jk}^{(q)})^2}
= \trace \bigg( 6 \frac{\partial \V}{\partial Z_{jk}^{(q)}} \mathbf R^2 \frac{\partial \V}{\partial Z_{jk}^{(q)}} \mathbf R \frac{\partial \V}{\partial Z_{jk}^{(q)}} \mathbf R \\
	- 3 \frac{\partial^2 \V}{\partial (Z_{jk}^{(q)})^2} \, \mathbf R^2 \frac{\partial \V}{\partial Z_{jk}^{(q)}} \mathbf R 
	- 3 \frac{\partial^2 \V}{\partial (Z_{jk}^{(q)})^2} \, \mathbf R \frac{\partial \V}{\partial Z_{jk}^{(q)}} \mathbf R^2 
	+ \frac{\partial^3 \V}{\partial (Z_{jk}^{(q)})^3} \, \mathbf R^2 \bigg) \,,
\end{multline*}
where, for $\kappa \in \{ 1,2,3 \}$, 
$$
\frac{\partial^\kappa \F_{n,t}}{\partial (Z_{jk}^{(q)})^\kappa} = \H_n^{(q)} \frac{\partial^\kappa (\Z_n^{(q)})_t^{\varepsilon_q}}{\partial (Z_{jk}^{(q)})^\kappa} \big( \widetilde\H_n^{(q)} \big)^*
$$
and
$$
\frac{\partial^\kappa (\Z_n^{(q)})_t^{\varepsilon_q}}{\partial (Z_{jk}^{(q)})^\kappa}
=
\sum_i \v_{i,1} \v_{i,2}^* \cdots \v_{i,2\kappa-1} \v_{i,2\kappa}^*
$$
where the sum is finite (with a number of summands which does not depend on~$n$),
$$
\v_{i,1},\hdots,\v_{i,2\kappa} \in \{ \A_{n,t} \e_k, \Z_n \A_{n,t} \Z_n^* \e_j, \A_{n,t} \Z_n^* \e_j, \Z_n \A_{n,t} \e_k, \A_{n,t} \e_k, \e_j, \e_k \}
$$
and 
$$
\| \v_{i,1} \|_2 \cdots \| \v_{i,2\kappa} \|_2 \leq t^{-(\kappa+1)/2} \,.
$$
Therefore, using H\"older's inequality, independence
as well as \eqref{eq:directproduct} and \eqref{eq:inverseproduct}
and proceeding as in \eqref{eq:B030},
we find that the conditional expectations 
given $X_{jk}^{(q)},Y_{jk}^{(q)}$
of $\partial g_{jk}^{(q)}(\theta) / \partial (Z_{jk}^{(q)})$
and $\partial^2 g_{jk}^{(q)}(\theta) / \partial (Z_{jk}^{(q)})^2$
are bounded by expressions of the form $C v^{-3} t^{-(l+1)}$ and $C v^{-4} t^{-3(l+1)/2}$,
and Conditions \eqref{B1} and \eqref{B2} are proved.
\end{proof}

\subsection{On Condition C}
\label{sub:condition-C}

Here we provide a number of lemmas which will be helpful in verifying Conditions C and \CS.
Recall that Condition \CS was introduced in Remark \ref{rem:C-simple}.

\begin{lemma}
\label{lemma:C-transfer}
For each $n \in \mynat$, let $\F_n$ and $\G_n$ be random matrices of dimension $n \times n$.
If the~matrices $\F_n$ and $\G_n$ satisfy Condition \CS,
then the~matrix~products $\F_n \G_n$ also satisfy Condition \CS.
\end{lemma}

\begin{proof} We use similar arguments as in the proof of Theorem 8.22 in \cite{GKT:2014}.
Condition (C0) follows from Corollary \ref{cor:horn}, Cauchy-Schwarz inequality,
and Condition (C0) for the matrices $\F_n$~and~$\G_n$.
Condition (C1) with $\alpha = 0$ follows from Theorem~\ref{thm:horn}
and the fact that the matrices $\F_n$~and~$\G_n$
satisfy Condition (C1) with $\alpha = 0$.
Thus, it remains to check Condition (C2) with $\alpha = 0$.
Suppose that $0 < \gamma < 1$ and $\delta_n \to 0$, \linebreak[2]
and set $n_1 := [n-n\delta_n]$, $n_2 := [n-n^\gamma]$ as~usual.
We will show that for $\pm \in \{ +,- \}$,
\begin{align}
\label{eq:plusminus}
\tfrac1n \sum_{n_1 \leq j \leq n_2} \log^{\pm} s_j(\F_n \G_n) \to 0 \quad \text{in probability} \,.
\end{align}
For the positive part, this follows from the fact that
the matrices $\F_n \G_n$ satisfy Condition (C0).
For the negative part, note that by~Theorem~\ref{thm:horn}, we have
$$
\prod_{j=k}^{n} s_j(\F_n \G_n) \geq \prod_{j=k}^{n} s_j(\F_n) \cdot \prod_{j=k}^{n} s_j(\G_n) 
$$
for $k=1,\hdots,n$. % which implies (with $k := \max \{ n_1, \min \{ j: s_j(\F_n \G_n) < 1 \} \}$) that
% $$
% \tfrac1n \sum_{j=n_1}^{n} \log^- s_j(\F_n \G_n) \leq \tfrac1n \sum_{j=n_1}^{n} \log^- s_j(\F_n) + \tfrac1n \sum_{j=n_1}^{n} \log^- s_j(\G_n) \,.
% $$
Thus, taking $k := \max \{ n_1, \min \{ j: s_j(\F_n \G_n) < 1 \} \}$
and using that the matrices $\F_n$ and $\G_n$ satisfy Conditions (C1) and (C2) with $\alpha = 0$,
we~get
$$
\tfrac1n \sum_{j=n_1}^{n} \log^- s_j(\F_n \G_n) \leq \tfrac1n \sum_{j=n_1}^{n_2} \log^- s_j(\F_n) + \tfrac1n \sum_{j=n_1}^{n_2} \log^- s_j(\G_n) + o_P(1) = o_P(1) \,,
$$
and the proof is complete.
\end{proof}

\begin{lemma}
\label{lemma:C-base}
For each $n \in \mynat$, let $\X_n = (\tfrac{1}{\sqrt{n}} X_{jk})_{jk=1,\hdots,n}$
be a matrix as in \eqref{eq:girko-ginibre},
where the entries $X_{jk}$ are independent random variables 
satisfying \eqref{eq:real-moments} -- \eqref{eq:uniform-integrability}.
Then the matrices $\X_n$ and $\X_n^{-1}$ satisfy Condition C.
\end{lemma}

\begin{proof}
To shorten notation, we omit the index $n$ throughout this proof. 

For the matrices $\X$, Condition C is checked in \cite{GT:2010b}
(in fact, it follows from the~relation $\ee\|\X\|_2^2 = n$
and from Lemmas \ref{lemma:C1} and \ref{lemma:C2}),
and for the matrices $\X^{-1}$, Condition C follows essentially
from the arguments in the proof of Theorem~8.22 in~\cite{GKT:2014}. 
For the convenience of the reader (and since we refer to the proof later),
let us provide some details for the matrices $\X^{-1}$.

\pagebreak[2]

For Condition (C0), see the proof of Theorem~8.22 in \cite{GKT:2014}. 
We now check Conditions (C1) and (C2) with $\alpha = 0$.
Condition (C1) follows from the relation $s_n(\X^{-1})= s_1^{-1}(\X)$
and Condition (C0) for the matrices $\X$.
For Condition~(C2), let $n_1 = [n - n \delta_n] \le [n - n^{1-\gamma}] = n_2$, 
where $\delta_n \ge n^{-\gamma}$, $\delta_n \to 0$.
Similarly as in~\eqref{eq:plusminus},
we consider the positive and negative part separately.
Since the matrices $\X$ satisfy Condition (C2) with $\alpha = 0$, 
we have, for $n$ large~enough,
$$
  \tfrac1n \sum_{j=n_1}^{n_2} \log^+ s_j(\X^{-1})
= \tfrac1n \sum_{j=n_1}^{n_2} \log^- s_{n-j+1}(\X)
\leq \tfrac1n \sum_{j=n_1}^{n_2} \log^- s_{j}(\X)
= o_P(1) \,.
$$
Moreover, suppose that the matrices $\X$ satisfy Condition (C0) with exponent $p$.
Then, since $\log x / x^p$ is decreasing in $x$ for $x \ge e^{1/p}$,
we have, for $n$ large~enough,
\begin{align*}
   \tfrac1n \sum_{j=n_1}^{n_2} \log^- & s_j(\X^{-1})
 = \tfrac1n \sum_{j=n_1}^{n_2} \log^+ s_{n-j+1}(\X) \\
&\le \tfrac1n \sum_{\substack{j=n_1,\hdots,n_2 \\ s_{n-j+1}(\X) \le \delta_n^{-1}}} \log^+ s_{n-j+1}(\X) 
   + \tfrac1n \sum_{\substack{j=n_1,\hdots,n_2 \\ s_{n-j+1}(\X)  >  \delta_n^{-1}}} \log^+ s_{n-j+1}(\X) \\
&\le \frac{n \delta_n - n^{1-\gamma}}{n} \log (\delta_n^{-1}) + \delta_n^{p} \log (\delta_n^{-1}) \, \tfrac1n \sum_{j=1}^{n} s_j^{p}(\X) \\
&\le \delta_n \log \delta_n^{-1} + \delta_n^{p} \log \delta_n^{-1} \myo_P(1)
= o_P(1) .
\end{align*}
Thus, % we have proved that 
the matrices $\X^{-1}$ satisfy Condition \CS.
We finally check Conditions (C1)~and~(C2) with $\alpha \ne 0$.
Here we may write 
$
\X^{-1} - \alpha \I = - \alpha (\X - \alpha^{-1} \I) \X^{-1},
$ 
and apply Lemma \ref{lemma:C-transfer} with $\F = \X - \alpha^{-1} \I$ and $\G = \X^{-1}$.
\end{proof}

\begin{remark*}
It follows from the preceding proof that if some matrices $\G_n$ satisfy Condition \CS,
then the inverse matrices $\G_n^{-1}$ satisfy Conditions (C1) and (C2) with $\alpha = 0$.
\end{remark*}

\pagebreak[1]

\begin{lemma}
\label{lemma:C-simple}
Let $\F_n = (\X_n^{(i_1)})^{\varepsilon_1} \cdots (\X_n^{(i_l)})^{\varepsilon_l}$,
where $l \in \mynat$, $i_1,\hdots,i_l \in \mynat$ \linebreak (not necessarily different)
and $\varepsilon_1,\hdots,\varepsilon_l \in \{ -1,+1 \}$ are fixed.
Then $\F_n$ satisfies Condition \CS.
\end{lemma}

\begin{proof}
By~Lemma~\ref{lemma:C-base},
the claim is true (even with the stronger Condition C) for~$l = 1$.
By~Lemma~\ref{lemma:C-transfer} and induction,
the claim remains true for~$l > 1$.
\end{proof}

\begin{lemma} 
\label{lemma:C++}
For each $n \in \mynat$, let $\X_n = (\tfrac{1}{\sqrt{n}} X_{jk})_{jk=1,\hdots,n}$
be a matrix as~in~\eqref{eq:girko-ginibre},
where the entries $X_{jk}$ are independent random variables 
satisfying \eqref{eq:real-moments} -- \eqref{eq:uniform-integrability}.
Furthermore, for each $n \in \mynat$, let $\A_n$, $\B_n$ and $\C_n$ be random matrices 
of dimension $n \times n$ such that $\A_n$, $\B_n$, $\C_n$ and $\X_n$ are independent. \pagebreak[2]
\begin{enumerate}[(a)]
\item
If the matrices $\A_n$ and $\B_n$ satisfy Condition \CS and the matrices $\C_n$ satisfy Condition $(C0)$,
then the~matrices $\A_n \X_n \B_n + \C_n$ satisfy Condition~C.
\item
If the matrices $\A_n$ and $\B_n$ satisfy Condition \CS and the matrices~$\C_n$ satisfy Condition~C
or $\C_n = \mathbf{0}$ for all $n \in \mynat$, then the matrices $\A_n \X_n^{-1} \B_n + \C_n$ satisfy Condition C.
\end{enumerate}
\end{lemma}

\begin{proof}[Proof]
To shorten notation, we omit the index $n$ throughout this proof.
First~of~all, let us note that if a sequence of random matrices $\G_n$ 
(with $\G_n$ of dimension $n \times n$) satisfies Condition (C0),
there exists some $L_\G > 0$ such~that
\begin{align}
\label{eq:polynomial-bound}
\lim_{n \to \infty} \pp(\|\G_n\| \geq n^{L_\G}) = 0 \,.
\end{align} 
In fact, if $p > 0$ is such that $\tfrac1n \sum_{k=1}^{n} s_k^p(\G_n)$
is bounded in~probability as $n \to \infty$ and $\varepsilon > 0$ is~arbitrary, it follows that
$$
     \limsup_{n \to \infty} \pp(s_1(\G_n) \ge n^{(1+\varepsilon)/p})
% \leq \limsup_{n \to \infty} \pp(s_1^p(\G_n) \ge n^{1+\varepsilon})
\leq \limsup_{n \to \infty} \pp(\tfrac{1}{n} \sum_{k=1}^{n} s_k^p(\G_n) \ge n^\varepsilon)
= 0 \,,
$$
so that the assertion holds for any $L_\G > 1/p$.

\pagebreak[1]

(a) Condition (C0) follows from Corollary \ref{cor:singularvalues},
Corollary \ref{cor:horn}, H\"older's inequality,
and the fact that the matrices $\A$, $\B$, $\C$ and $\X$ satisfy Condition (C0).
To prove Conditions (C1) and (C2), we~use the factorization
$$
\A \X \B + \C - \alpha \I = \A (\X + \A^{-1} (\C - \alpha \I) \B^{-1}) \B \,.
$$
Then, similarly as in the proof of Lemma \ref{lemma:C-transfer},
it remains to be checked that for each of the three factors $\M_n$ 
on the right-hand side, we have, for some $Q > 0$,
$$
\pp(s_n(\M_n) \leq n^{-Q}) = o(1)
\quad\text{and}\quad
\tfrac1n \sum_{n_1 \le n \le n_2} \log^- s_j(\M_n) = o_P(1) \,.
$$
For $\A$ and $\B$, this is true by assumption.
For $\X + \A^{-1} (\C - \alpha \I) \B^{-1}$,
this follows from Lemmas \ref{lemma:C1x} and \ref{lemma:C2x}.
More precisely, if the matrices $\A$~and~$\B$ satisfy Con\-dition (C1) with $\alpha = 0$ and $Q > 0$,
and the~matrices $\C$ satisfy \eqref{eq:polynomial-bound} with $L_\C > 0$,
we have
$\pp(s_1(\A^{-1} (\C - \alpha \I) \B^{-1}) > n^{2Q+L_\C}) \to 0$
by Theorem \ref{thm:horn}.
Thus, we may use Lemmas \ref{lemma:C1x} and \ref{lemma:C2x}
conditionally on~$\A,\B,\C$,
and on the set of~probability $1 + o(1)$
where $s_1(\A^{-1} (\C - \alpha \I) \B^{-1}) \leq n^{2Q+L_\C}$.

\pagebreak[1]

(b) We consider only the case that the matrices $\C$ satisfy Condition C,
leaving the simpler case $\C = \mathbf 0$ to the reader.
Similarly as above, Condition (C0) follows 
from Corollary \ref{cor:singularvalues},
Corollary \ref{cor:horn}, H\"older's inequality,
and the fact that the matrices $\A$, $\B$, $\C$ and~$\X^{-1}$ satisfy Condition~(C0).
To prove Conditions (C1) and (C2), we~use the factorization
$$
\A \X^{-1} \B + \C - \alpha \I = \A \X^{-1} (\B(\C - \alpha \I)^{-1}\A + \X) \A^{-1} (\C - \alpha \I) \,.
$$
Again, similarly as in the proof of Lemma \ref{lemma:C-transfer},
it remains to be checked that for each of the five factors $\M_n$ 
on the right-hand side, we have, for some $Q > 0$,
$$
\pp(s_n(\M_n) \leq n^{-Q}) = o(1)
\quad\text{and}\quad
\tfrac1n \sum_{n_1 \le n \le n_2} \log^- s_j(\M_n) = o_P(1) \,.
$$
But this is true
(i) by assumption,
(ii) by Lemma \ref{lemma:C-base},
(iii) by Lemmas \ref{lemma:C1x}~and~\ref{lemma:C2x} (applied conditionally on $\A$, $\B$, $\C$),
(iv) by the remark below Lemma \ref{lemma:C-base},
(v) by~assumption.
\end{proof}

\subsection{Proof of Theorem \ref{thm:limit-theorem}}
\label{sub:main-proof}

After these preparations, we may turn to the proof of Theorem \ref{thm:limit-theorem}.
Given a sequence of random matrices $(\G_n)_{n \in \mynat}$, we write 
$\nu(\G_n)$ for the singular value distributions,
$\mu(\G_n \G_n^*)$ for the squared singular value distributions,
$\mys \nu(\G_n)$ for the symmetrized singular value distributions,
and $\nu_\G$, $\mu_{\G \G^*}$ and $\mys \nu_\G$ for the corresponding 
weak limits in~prob\-ability (if existent).

\pagebreak[3]

Let us first consider the singular value distributions.
We will first use induction on $l$ to prove the claim for the case $m=1$
and then use induction on $m$ to prove the claim for the case $m>1$.
More precisely, we~will show the following:
\begin{align}
\label{eq:1000}
\begin{array}{l}
\text{The matrices $\F_n(\X)$ from \eqref{eq:sums-of-products} satisfy Conditions A and B,} \\
\text{and for any $t > 0$, the singular value distributions} \\
\text{of the matrices $\F_{n,t}(\X)$ converge weakly in probability} \\
\text{to the probability measure $\nu_{t}$ on $(0,\infty)$ with symmetrization} \\
\text{$\mys \nu_{t} = \big( \myq^{-1} ( \gamma_t^{\varepsilon_1} \boxtimes \cdots \boxtimes \gamma_t^{\varepsilon_l} ) \big)^{\boxplus m}$.}
\end{array}
\end{align}
Indeed, by Condition~A, we may then let $t \to 0$ to get the limiting singular value
distribution of the matrices $\F_n(\X)$.

\pagebreak[2]

\emph{Products of independent random matrices.}
For $\F_n(\X) = \X_n$ and $\F_n(\X) = \X_n^{-1}$, Conditions A and B follow from 
Lemmas \ref{lemma:A} and \ref{lemma:B}, respectively.
It follows from the results in Section \ref{sec:RMT} that,
for $t > 0$ and $\varepsilon \in \{ -1,+1 \}$,
we have $\mu(\X_{n,t}^{\varepsilon} (\X_{n,t}^\varepsilon)^*) \to \gamma_t^\varepsilon$.
Thus, \eqref{eq:1000} is true for $l = 1$.

Now let $l > 1$, suppose that \eqref{eq:1000} holds for any $(l-1)$-fold product $\G_n$,
and let $\F_n$ be an $l$-fold product.
Then we have the representation $\F_n(\X) = \X_n^\varepsilon \G_n(\X)$,
where $\varepsilon = +1$ or $\varepsilon = -1$
and $\X_n$ and $\G_n(\X)$ are independent.
The inductive hypo\-thesis ensures that $\mu(\G_n(\X) \G_n^*(\X)) \to \mu_{\G \G^*}$, 
a non-random probability measure on $(0,\infty)$.
It therefore follows from Lemmas \ref{lemma:A} and \ref{lemma:B}
that the matrices $\F_n(\X)$ satisfy Conditions A and B.
Now, for any $t > 0$, the matrices $\Y_{n,t}^\varepsilon$ and $\G_{n,t}(\Y)$
are independent bi-unitary invariant matrices
with 
$$
\mu(\Y_{n,t}^\varepsilon (\Y_{n,t}^\varepsilon)^*) \to \gamma_t^{\varepsilon}
\quad\text{and}\quad
\mu(\G_{n,t}(\Y) \G_{n,t}^*(\Y)) \to \mu_{\G(t)\G(t)^*} \,, 
$$
respectively. Therefore, by asymptotic freeness (see Proposition \ref{prop:asymptoticfreeness}\,(a)),
$$
\mu(\F_{n,t}(\Y) \F_{n,t}^*(\Y)) \to \gamma_t^{\varepsilon} \boxtimes \mu_{\G(t)\G(t)^*} \,.
$$ 
Thus, by Theorem \ref{thm:universality}\,(a),
\eqref{eq:1000} holds for the matrices $\F_n(\X)$ as well.

Hence, by induction on $l$, we come to the conclusion
that \eqref{eq:1000} holds for any product of independent matrices
(i.e. for the case $m = 1$).

\emph{Sums of products of independent random matrices.}
We have just proved \eqref{eq:1000} for $m = 1$.
Now let $m > 1$, suppose that \eqref{eq:1000} holds 
for any $(m-1)$-fold sum $\C_n$ of products of independent random matrices,
and let $\F_n$ be an $m$-fold sum of products of independent random matrices.
Then we have the representation $\F_n(\X) = \G_n(\X) + \C_n(\X)$,
where $\G_n(\X) = \X_n^{\varepsilon} \B_n(\X)$ (possibly with $\B_n(\X) = \I_n$) 
\linebreak and $\X_n$, $\B_n(\X)$ and $\C_n(\X)$ are independent.
The result for the case $m=1$ and the inductive hypothesis ensure
that $\nu(\B_n(\X)) \to \nu_\B$ and $\nu(\C_n(\X)) \to \nu_\C$, respectively,
where $\nu_\B$ and $\nu_\C$ are non-random probability measures on $(0,\infty)$.
It therefore follows from Lemmas \ref{lemma:A} and \ref{lemma:B}
that the matrices $\F_n(\X)$ satisfy Conditions A and B.
Moreover, for any $t > 0$, the matrices $\G_{n,t}(\Y)$ and $\C_{n,t}(\Y)$ 
are independent bi-unitary invariant matrices
with 
$$
\mys\nu(\G_{n,t}(\Y)) \to \mys\nu_{\G(t)}
\quad\text{and}\quad
\mys\nu(\C_{n,t}(\Y)) \to \mys\nu_{\C(t)} \,
$$
respectively.
Therefore, by asymptotic freeness (see Proposition \ref{prop:asymptoticfreeness}\,(c)),
$$
\mys\nu(\F_{n,t}(\Y)) \to \mys\nu_{\C(t)} \boxplus \mys\nu_{\G(t)} \,.
$$ 
Thus, by Theorem \ref{thm:universality}\,(a),
\eqref{eq:1000} holds for the matrices $\F_n(\X)$ as~well.

Hence, by induction on $m$, we come to the conclusion
that \eqref{eq:1000} holds for any sum of products of independent matrices
(i.e. for the case $m > 1$).

Let us now consider the eigenvalue distributions. 
To begin with, using Lemma \ref{lemma:C++},
we may check by~induction on $m$ that the matrices $\F_n(\X)$ also satisfy Condition~C. 
Therefore, as the matrices $\F_n(\X)$ satisfy Conditions A, B and C,
we may use Theorem \ref{thm:universality}\,(b), and it remains to determine
the limiting eigenvalue distributions in the~Gaussian case, i.e.\@ for the matrices $\F_n(\Y)$.
Here, it~follows by asymptotic freeness (see Proposition \ref{prop:asymptoticfreeness}\,(d)) that 
$\mys \nu(\F_{n,t}(\Y) - \alpha \I_n) \to \mys \nu_{t,\alpha}$ $:= (\mys \nu_t) \boxplus B(\alpha)$.
Letting $t \to 0$ and using Condition~A, it further follows that
$\mys \nu(\F_{n}(\Y) - \alpha \I_n) \to \mys \nu_{\alpha} := (\mys \nu) \boxplus B(\alpha)$,
where $\nu$ is the probability measure described in the theorem.
Now use Theorem \ref{thm:main}.
\qed

\subsection{Rigorous Proof of Theorems \ref{thm:new} and \ref{thm:newpower}}
\label{sub:rigor}

\begin{proof}[Rigorous Proof of Theorem \ref{thm:new}]
By Theorem \ref{thm:limit-theorem}, the limiting eigenvalue distri\-butions
of the matrices $\F_n^{(0)}$ and $m^{-(l+1)/2} (\F_n^{(1)}+\cdots+\F_n^{(m)})$
in Theorem \ref{thm:new} are given by 
$$
\myh\Big(\myq^{-1}(\gamma_1^{-1} \boxtimes \gamma_1^{\boxtimes l})\Big)
\quad\text{and}\quad
\myh\Big( \myd_{m^{-(l+1)/2}} (\myq^{-1}(\gamma_1^{-1} \boxtimes \gamma_1^{\boxtimes l}))^{\boxplus m} \Big) \,,
$$
respectively. 
Now, similarly as in the formal proof of Theorem \ref{thm:new}, 
we find that
$$
\myq^{-1}(\gamma_1^{-1} \boxtimes \gamma_1^{\boxtimes l}) = \sigma_s(\tfrac{2}{l+1})
$$
by comparing the $S$-transforms of the two measures,
which concludes the proof.
\end{proof}

\begin{proof}[Rigorous Proof of Theorem \ref{thm:newpower} (Sketch)]
Here we need an analogue of Theorem \ref{thm:limit-theorem} for \emph{certain}
products of powers of independent Girko--Ginibre matrices and their inverses.
More precisely, we now consider random matrices of the form
\begin{align}
\label{eq:sums-of-products-of-powers}
\F_n := \sum_{q=1}^{m} \F_n^{(q)} := \sum_{q=1}^{m} \prod_{r=1}^{k} ((\X_n^{((q-1)l+r)})^{\varepsilon_{r}})^{l_r} \,,
\end{align}
where $m,k \in \mynat$, $\varepsilon_1,\hdots,\varepsilon_k \in \{ -1,+1 \}$ 
and $l_1,\hdots,l_k \in \mynat$ are fixed,
\begin{align}
\label{eq:extra-condition}
\text{for some $r=1,\hdots,k$, we have $l_r = 1$,}
\end{align}
and the~$\X_n^{(q)}$ are independent Girko-Ginibre matrices as in the introduction.
Then, similarly as in the proof of Theorem~\ref{thm:limit-theorem},
we need to verify Conditions A, B and C:

\medskip

\noindent\textit{Condition A.}
Here we can regularize the matrices $(\X_n^{-1})^l$ by means of $((\X_n)_t^{-1})^l$
(i.e.\@ each factor in the power is regularized individually) and invoke Lemma~\ref{lemma:A}.
It is important here that in Lemma~\ref{lemma:A}, 
the matrices $\A_n$, $\B_n$ and~$\C_n$ need not be independent of $\X_n$;
see Remark \ref{remark:no-independence}.

\medskip

\noindent\textit{Condition B.}
Here we may extend Lemma \ref{lemma:B} to products of powers of independent Girko--Ginibre matrices, 
using similar arguments as in Sections 8.1.3 and 8.1.4 in~\cite{GKT:2014}.

\medskip

\noindent\textit{Condition C.}
Under the extra condition \eqref{eq:extra-condition},
it follows from Lemma \ref{lemma:C++} (applied with $\X = \X^{(r)}$)
and by induction on $m$ that the matrices $\F_n$ satisfy Condition~C. 
(Unfortunately, without condition \eqref{eq:extra-condition},
Lemma~\ref{lemma:C++} does not allow us to draw this conclusion in general,
even though we would expect that Condition~C continues to hold in this case.)

\pagebreak[2]
\medskip

After these considerations, it is straightforward to extend Theorem \ref{thm:limit-theorem}
to sums of products of powers satisfying \eqref{eq:extra-condition}.
It follows that the limiting eigenvalue distributions
of the matrices $\F_n^{(0)}$ and $m^{-(l+1)/2} (\F_n^{(1)}+\cdots+\F_n^{(m)})$
in Theorem \ref{thm:newpower} are given by 
$$
\myh\Big(\myq^{-1}(\gamma_1^{-1} \boxtimes \gamma_1^{\boxtimes l_1} \boxtimes \cdots \boxtimes \gamma_1^{\boxtimes l_k})\Big)
\quad\text{and}\quad
\myh\Big( \myd_{m^{-(l+1)/2}} (\myq^{-1}(\gamma_1^{-1} \boxtimes \gamma_1^{\boxtimes l_1} \boxtimes \cdots \boxtimes \gamma_1^{\boxtimes l_k})^{\boxplus m}) \Big) \,,
$$
respectively. 
Now, similarly as in the formal proof of Theorem \ref{thm:newpower}, 
we find that
$$
\myq^{-1}(\gamma_1^{-1} \boxtimes \gamma_1^{\boxtimes l_1} \boxtimes \cdots \boxtimes \gamma_1^{\boxtimes l_k}) = \sigma_s(\tfrac{2}{l+1})
$$
by comparing the $S$-transforms of the two measures, which concludes the proof.
\end{proof}

\medskip

\section{Appendix: Auxiliary Results}
\label{sec:auxiliary-results}

\subsection{Inequalities for Singular Values} 
In this section we collect a number of results from \cite[Section 3.3]{HoJo} 
which we use to verify Conditions (C0) -- (C2).
Throughout this section, we assume that $\A$ is a square matrix of dimension $n \times n$
with eigenvalues $|\lambda_1(\A)| \geq \cdots \geq |\lambda_n(\A)|$
and singular values $s_1(\A) \geq \cdots \geq s_n(\A)$.
We usually state the results for the~largest singular values,
but using the relation $s_j(\A^{-1}) = s_{n-j+1}^{-1}(\A)$, $j=1,\hdots,n$,
it is immediate that similar results hold for the~smallest singular values.
Also, it is easy to see that Theorem \ref{thm:horn} and its~corollary
extend to matrix products with more than two factors.

\begin{theorem}[Horn]
\label{thm:horn}
For all $k=1,\hdots,n$,
$
\prod_{j=1}^{k} s_j(\A\B) \le \prod_{j=1}^{k} s_j(\A) s_j(\B) \,,
$
with equality for $k = n$.
\end{theorem}

\begin{corollary}
\label{cor:horn}
For all $p > 0$ and all $k=1,\hdots,n$, we have
$
\sum_{j=1}^{k} (s_j(\A\B))^p \le \sum_{j=1}^{k} (s_j(\A) s_j(\B))^p \,.
$
\end{corollary}

For the next lemma, let us make the convention that $s_j(\A) := 0$ for $j > n$.

\begin{lemma}
\label{lemma:singularvalues}
For all $j,k=1,\hdots,n$, we have
$s_{j+k-1}(\A+\B) \leq s_j(\A) + s_k(\B)$
and
$s_{j+k-1}(\A \cdot \B) \leq s_j(\A) \cdot s_k(\B)$.
\end{lemma}

\begin{corollary}
\label{cor:singularvalues}
For all $p > 0$,
$
\sum_{j=1}^{n} s_j^p(\A+\B) \leq C_p \left( \sum_{j=1}^{n} s_j^p(\A) + \sum_{j=1}^{n} s_j^p(\B) \right) \,,
$
where $C_p$ is a positive constant depending only on $p$.
\end{corollary}

\pagebreak[2]

\subsection{Bounds on Small Singular Values} \
In this section we cite some \linebreak stochastic bounds for small singular values from the literature.
We always assume that the matrices $\X_n$ are Girko-Ginibre matrices
as in \eqref{eq:girko-ginibre} -- \eqref{eq:uniform-integrability}.

\begin{lemma}[{\cite[Theorem 4.1]{GT:2010a}}]
\label{lemma:C1}
 Suppose that conditions \eqref{eq:girko-ginibre} -- \eqref{eq:uniform-integrability} hold. 
 Then, for any fixed $\alpha \in \mathbb{C}$,
 there exist positive constants $A$ and $B$ such that
 \begin{equation}\notag
  \Pr\{s_n(\X_n-\alpha\I_n)\le n^{-A}\}\le n^{-B}.
 \end{equation}
\end{lemma}

\begin{lemma}[{\cite[Lemma 5.2]{GT:2010b}}]
\label{lemma:C2}
 Suppose that conditions \eqref{eq:girko-ginibre} -- \eqref{eq:uniform-integrability} hold. 
 Then, for any fixed $\alpha \in \mathbb{C}$,
 there exists a constant $0<\gamma<1$ such that for any sequence $\delta_n \to0 $, 
 \begin{equation}\notag
  \lim_{n\to\infty}\frac1n\sum_{{n_1}\le j\le{ n_2}} |\log s_j(\X_n-\alpha\I_n)| =0 \quad \text{almost surely},
 \end{equation}
 with $n_1=[n-n\delta_n]+1$ and $n_2=[n-n^{\gamma}]$.
\end{lemma}

More generally, these results hold with $\alpha \I_n$ replaced by $\M_n$,
where $(\M_n)_{n \in \mynat}$ is a sequence of deterministic matrices
which is polynomially bounded in operator norm;
see Section~5 in \cite{GT:2010b} for details:

\begin{lemma}
\label{lemma:C1x}
 Suppose that conditions \eqref{eq:girko-ginibre} -- \eqref{eq:uniform-integrability} hold. 
 Then, for any fixed $K>0$ and $L>0$,
 there exist positive constants $A$ and $B$ such that for any non-random matrix $\M_n$
 with $\|\M_n\|_2 \leq Kn^{L}$, we have
 \begin{equation}\notag
  \Pr\{s_n(\X_n-\M_n)\le n^{-A}\}\le n^{-B}.
 \end{equation}
\end{lemma}

\begin{lemma}
\label{lemma:C2x}
 Suppose that conditions \eqref{eq:girko-ginibre} -- \eqref{eq:uniform-integrability} hold. 
 Then, for any fixed $K>0$ and $L>0$,
 there exists a constant $0<\gamma<1$ such that for any non-random matrices $\M_n$
 with $\|\M_n\|_2 \leq Kn^{L}$ and for any sequence $\delta_n \to0 $, 
 \begin{equation}\notag
  \lim_{n\to\infty}\frac1n\sum_{{n_1}\le j\le{ n_2}} \log^- s_j(\X_n-\M_n)=0 \quad \text{almost surely},
 \end{equation}
 with $n_1=[n-n\delta_n]+1$ and $n_2=[n-n^{\gamma}]$.
\end{lemma}

\end{document}